\documentclass[reqno,12pt]{amsart}
\usepackage[colorlinks=true,citecolor=magenta,linkcolor=magenta]{hyperref}
\usepackage{amsmath,amssymb}
\usepackage{bm} 
\usepackage{enumitem} 
\usepackage[margin=1in]{geometry}
\usepackage{color}
\newcommand{\uD}{\underline{D}}

\newcommand{\bb}{\mathfrak}
\newcommand{\measv}{\mathcal{L}^{d+1}}
\newcommand{\measb}{\mathcal{L}^{d}}
\newcommand{\intOO}{\int_{\Omega}}
\newcommand{\intGG}{\int_{\Gamma}}

\title[Bulk-surface systems on evolving domains]{Bulk-surface systems on evolving domains}

\author{Diogo Caetano, Charles M. Elliott, Bao Quoc Tang}
\address{Diogo Caetano \hfill\break
	Mathematics Institute, University of Warwick, Coventry CV4 7AL, United Kingdom}
\email{Diogo.Caetano@warwick.ac.uk}
\address{Charles Elliott \hfill\break
	Mathematics Institute, University of Warwick, Coventry, CV4 7AL, United Kingdom}
\email{C.M.Elliott@warwick.ac.uk}
\address{Bao Quoc Tang \hfill\break
	Institute of Mathematics and Scientific Computing, University of Graz, Heinrichstrasse 36, 8010 Graz, Austria}
\email{quoc.tang@uni-graz.at} 
\newtheorem{theorem}{Theorem}[section]
\newtheorem{lemma}[theorem]{Lemma}
\newtheorem{proposition}[theorem]{Proposition}	

\newtheorem{corollary}[theorem]{Corollary}
\newtheorem{remark}[theorem]{Remark}

\newcommand{\ou}{\overline{u}}
\newcommand{\ow}{\overline{w}}
\newcommand{\oz}{\overline{z}}	
	
\newcommand{\R}{\mathbb R}
\newcommand{\dO}{\delta_\Omega}
\newcommand{\dG}{\delta_\Gamma}
\newcommand{\dGG}{\delta_{\Gamma}'}
\newcommand{\GG}{\Gamma}
\newcommand{\GGG}{\mathcal G}
\renewcommand{\O}{\Omega}
\newcommand{\pa}{\partial}
\newcommand{\VV}{\mathbf{V}}
\newcommand{\na}{\nabla}
\newcommand{\naG}{\nabla_{\GG}}
\newcommand{\J}{\mathbf{J}}

\newcommand{\dV}{\mathrm{d}x\mathrm{ }}

\newcommand{\dx}{\mathrm{d}x}
\newcommand{\dt}{\mathrm{d}t}
\newcommand{\intG}{\int_{\GG(t)}}
\newcommand{\uinf}{u_{\infty}}
\newcommand{\winf}{w_{\infty}}
\newcommand{\zinf}{z_{\infty}}

\newcommand{\intO}{\int_{\Omega(t)}}
\newcommand{\intT}{\int_0^T}
\newcommand{\wt}{\widetilde}
\newcommand{\wh}{\widehat}

\numberwithin{equation}{section}
\begin{document}
	\subjclass[2020]{35A01, 35K57, 35K58, 35Q92, 35R01, 35R37}
	\keywords{Volume-surface reaction-diffusion systems; Global solutions; Evolving domains; Convergence to equilibrium; Entropy method}
	\begin{abstract}		
		Bulk-surface systems in evolving domains are studied. Such problems appear typically from modelling receptor-ligand dynamics in biological cells. Our first main result is the global existence and boundedness of solutions in all dimensions. This is achieved by proving $L^p$-maximal regularity of parabolic equations and duality methods in moving surfaces, which are of independent interest. The second main result is the large time dynamics where we show, under the assumption that the volume/area of the moving domain/surface is unchanged and that the material velocities are decaying for large time, that the solution converges to a unique spatially homogeneous equilibrium. The result is proved by extending the entropy method to bulk-surface systems in evolving domains.
	\end{abstract}
	
	\maketitle
	\tableofcontents
	
	\section{Introduction}
	
	\subsection{Ligand-receptor systems in evolving domains}
	For each $t\in [0,\infty)$, let $D(t) \in \R^{d+1}$ be a Lipschitz domain containing a $C^2$-hypersurface $\GG(t)$ which separates $D(t) = I(t)\cup \O(t)$ into an interior region $I(t)$ and an exterior (Lipschitz) domain $\O(t)$.  The points in  $\GG(t)$  and $\Omega(t)$ are subject to velocity fields   $\VV_\GG$  and $\VV_{\Omega}$ respectively and we denote $\Omega_0:=\Omega(0)$ and $\Gamma_0:=\Gamma(0)$.
	Denoting by $\nu(t)$ the outward pointing unit normal to $\partial \Omega(t)$, we suppose that the surface $\GG(t)$ and the outer boundary $\pa D(t)$ both evolve in time with normal kinematic velocity fields  
	\begin{equation}\label{eq_V_V0}
	\mathbf V=(\VV_\GG|_{\GG}\cdot \nu(t))\nu(t) \quad\text{ and }\quad \mathbf V_0, \quad \text{ respectively}. 
	\end{equation}
	
	We assume that these moving domains are parameterised by a flow $\Phi: [0,T]\times \R^{d+1}\to \R^{d+1}$ satisfying
	\begin{enumerate}[label=($\Phi$\theenumi), ref=$\Phi$\theenumi]
		\item\label{Phi1} $\Phi_t:= \Phi(t,\cdot): \overline{\O_0} \to \overline{\O(t)}$ is a $C^2$-diffeomorphism with $\Phi_t(\GG_0) = \GG(t)$.
		\item\label{Phi2} $\Phi_t$ solves the ODE
		\begin{equation*}
		\begin{gathered}
		\dfrac{d}{dt}\Phi_t(\cdot) = \mathbf{V}_p(t,\Phi_t(\cdot)),\\
		\Phi_0 = \mathrm{Id},
		\end{gathered}
		\end{equation*}
		in which $\VV_p: [0,T]\times \R^{d+1}\to \R^{d+1}$ is a continuously differentiable velocity and satisfies
		\begin{equation}\label{eq_Vp}
		(\VV_p|_{\GG}\cdot \nu(t))\nu(t) = \VV \qquad \text{ and } \qquad (\VV_p|_{\pa D}\cdot \nu(t))\nu(t) = \VV_0.
		\end{equation}
	\end{enumerate}
	
	The classical parametric (or strong) material derivatives on $\O(t)$ and $\GG(t)$ are defined as follows
	\begin{equation*}
	\pa_{\O}^{\bullet} u(t) = \pa_t u(t) + \nabla u(t)\cdot \VV_p(t) \quad \text{ and } \quad \pa_{\GG}^{\bullet} v(t) = \pa_t^\circ v(t) + \nabla_{\GG(t)}v(t)\cdot \VV_p(t)
	\end{equation*}
	where $\pa_t^\circ v$ is the {\it normal} time derivative on $\Gamma(t)$, that is, the rate of change on trajectories moving in the normal direction on $\Gamma(t)$.
	For convenience, we will simplify notation and denote by $\dot{u}$ and $\dot{v}$ the parametric material derivatives,
	\begin{equation*}
	\dot u(t) = \partial_{\Omega}^{\bullet} u(t), \quad \dot v(t) = \partial_{\Gamma}^{\bullet} v(t),
	\end{equation*}
	where the corresponding domains should be clear. We set
	\begin{equation*}
	\J_\O:= \VV_\O  - \VV_p, \quad \J_\GG:= \VV_\GG - \VV_p, \quad \text{ and } \quad j:= (\VV_\O - \VV_\GG)\cdot \nu(t) = \J_\O|_{\GG}\cdot \nu(t).
	\end{equation*}
	The study of this paper has its primary application to the following system
	\begin{equation}\label{sys}
	\left\{
	\begin{aligned}
	\dot{u} + u\na\cdot \VV_p + \na\cdot (\J_\O u)  - \dO \Delta u&= 0, &&\text{ in } \O(t),\\
	\dO \na u \cdot \nu - uj &= r(u,w,z), && \text{ on } \GG(t),\\
	\dot{w} + w\naG\cdot \VV_p + \naG \cdot (\J_\GG w)  - \dG \Delta_\GG w&= r(u,w,z), &&\text{ on } \GG(t),\\
	\dot{z} + z\naG \cdot \VV_p + \naG \cdot (\J_\GG z)  - \dGG \Delta z&= -r(u,w,z), &&\text{ on } \Gamma(t),\\
	u(x,0) &= u_0(x), &&\text{ for } x\in\O_0,\\
	z(x,0) = z_0(x), \; w(x,0) &= w_0(x) , &&\text{ for } x\in \GG_0.
	\end{aligned}
	\right.
	\end{equation}
	The diffusion coefficients $\delta_\Omega, \delta_\Gamma, \delta_{\Gamma'}$ are assumed to be positive. The initial data are nonnegative and bounded.
	For simplicity we omit the dependence on $t$ and write $\nu$, $\naG$ and $\Delta_\GG$ instead of outward normal vector $\nu(t)$, the tangential gradient $\nabla_{\GG(t)}$ and the Laplace-Beltrami operator $\Delta_{\GG(t)}$, which are all time-dependent. In each reaction diffusion equation the operator arises from conservation with a diffusive and advective flux with the advection velocity being given by the velocities  $\VV_{\Omega}$ and $\VV_\GG$, to which we therefore refer as material velocities. The reaction term $r(u,w,z)$ is defined through mass exchange with an application of the mass action law, more precisely
	\begin{equation*}
	r(u,w,z) = \frac{1}{\delta_K'}z - \frac{1}{\delta_K}uw.
	\end{equation*}
	The study of system \eqref{sys} is motivated from a
	biological application to receptor-ligand dynamics where $u$ represents the concentrations of ligands in $\Omega$, which acts as the extracellular volume, whilst $w$ and $z$ are the concentrations of surface receptors and of ligand-receptor complexes formed by the binding of $u$ with $w$ on $\Gamma$ (see more details in \cite[Section 1.3]{alphonse2018coupled}). Formally, solutions to \eqref{sys} satisfy the conservation laws
	\begin{equation}\label{law1}
	\intO u(x,t)  + \intG z(x,t)  = M_1:= \int_{\O_0} u_0(x)  + \int_{\GG_0} z_0(x) 
	\end{equation}
	and
	\begin{equation}\label{law2}
	\intG w(x,t)  + \intG z(x,t)  = M_2:= \int_{\GG_0} w_0(x)  + \int_{\GG_0} z_0(x) .
	\end{equation}
	
	\medskip
	This paper shows the global existence of a bounded solution \textit{in all dimensions} of the system \eqref{sys} and its convergence to equilibrium. In the course of doing that, we prove the $L^p$-maximal regularity and duality estimates for problems on moving surfaces, which are of independent interest.
	
	\subsection{State of the art}
	Problem \eqref{sys} is usually called a bulk-surface reaction-diffusion system on moving domains. On the one hand, the study of equations on time-dependent domains goes back to the seventies to the work of Fujita \cite{fujita1970existence}. Since then, such problems have been investigated extensively due to their applications in mechanics, biology, etc. The traditional approach is to transform the system into one on a fixed reference domain with time-dependent coefficients and investigate it using classical tools from partial differential equations. Due to this, the study is usually problem-dependent. Another approach is to treat the system directly by developing parallel analytical mechanisms on the moving domains, see e.g. \cite{alphonse2015abstract,meier2008note,vierling2014parabolic}. The latter turns out to be useful as it allows to investigate classes of problems thanks to the unified framework. On the other hand, problems involving the coupling between a volume and its boundary have attracted a lot of attention recently due to its appearance in modelling systems from asymmetric stem cell division \cite{fellner2016quasi,fellner2018well}, in modelling receptor-ligand dynamics in cell biology \cite{garcia2014mathematical,alphonse2018coupled}, in crystal growth \cite{kwon2001modeling,yeckel2005computer}, in population modelling \cite{berestycki2013influence,berestycki2015effect}, in chemistry with fast sorption \cite{souvcek2019continuum,augner2020analysis}, or in fluid mechanics \cite{mielke2013thermomechanical,glitzky2013gradient}, and so on. The bulk-surface coupling characteristic makes the study of such problems challenging, see e.g. \cite{fellner2018well,disser2020global,sharma2016global,sharma2017uniform,hausberg2018well}.
	
	\medskip
	The study of bulk-surface systems on moving domains is much less extensive comparing to its counterpart in fixed domains. Recent related works include \cite{alphonse2018coupled,elliott2017coupled,caetano2021cahn,bogosel2021propagation,hansbo2016cut}. In particular, in \cite{alphonse2018coupled} the second author and collaborators  studied the system \eqref{sys}
	under the restriction $d \leq 3$ and showed that \eqref{sys} has a nonnegative weak solution $(u, w, z)$ with 
	\begin{equation*}
	\|w\|_{L^{\infty}_{L^\infty(\GG)}} \leq C(T).
	\end{equation*}
	If the dimension is further restrictive $d\leq 2$, then one has also the bounds on $z$ and $u$, i.e.
	\begin{equation*}
	\|z\|_{L^{\infty}_{L^\infty(\GG)}} \leq C(T) \quad \text{ and } \|u\|_{L^{\infty}_{L^\infty(\O)}} \leq C(T).
	\end{equation*}
	(See the definition of function spaces in the next subsection). Their proof utilises the De Giorgi method, which leads  to the restriction on the dimensions. Due to the technique used therein, the question of boundedness in higher dimensions remains open even in the non-moving case. Concerning the dynamics of \eqref{sys}, it was shown in \cite{alphonse2018coupled} \textit{in fixed domains, i.e. $\Omega(t) \equiv \Omega_0$, $\Gamma(t) \equiv \Gamma_0$, for all $t\geq 0$}, that the solution converges exponentially to equilibrium.
	
	\subsection{Main results}
	The \textbf{\textit{first main result}} of the current paper is to prove that the solution to \eqref{sys} is bounded in any dimension, hence removing the restriction on the dimensions required in \cite{alphonse2018coupled}. Our main tools are $L^p$-maximal regularity and duality methods in evolving surfaces, which are of independent interest. These techniques are developed and used by many other authors in the case of cylindrical domains (see e.g. \cite{denk2003mathcal,CDF14,pierre2010global} and references therein). As the \textbf{\textit{second main result}}, we prove the convergence to equilibrium for the system \eqref{sys} {\it when the material velocities $\VV_\Omega$ and $\VV_\Gamma$ decay as $t\to\infty$}. This relaxes the fixed domain assumption in \cite{alphonse2018coupled} and provides, to the best of our knowledge, the first convergence to equilibrium for systems in moving domains. In order to do that, we refine the so-called entropy method to apply it to our situation.
	
	\medskip
	
	For the global existence of bounded solutions, we consider a more general system which contains \eqref{sys} as a special case. More precisely, we consider the following general system 
	\begin{equation}\label{sys_extend}
	\left\{
	\begin{aligned}
	\dot{u} + u\na\cdot \VV_p + \na\cdot (\J_\O u)  - \dO \Delta u&= 0, &&\text{ in } \O(t),\\
	\dO \na u \cdot \nu - uj &= f_1(u,w,z), && \text{ on } \GG(t),\\
	\dot{w} + w\naG\cdot \VV_p+ \naG \cdot (\J_\GG w)  - \dG \Delta_\GG w &= f_2(u,w,z), &&\text{ on } \GG(t),\\
	\dot{z} + z\naG \cdot \VV_p + \naG \cdot (\J_\GG z) - \dGG \Delta_\GG z &= f_3(u,w,z), &&\text{ on } \Gamma(t),\\
	u(x,0) &= u_0(x), &&\text{ for } x\in\O_0,\\
	z(x,0) = z_0(x), w(x,0) &= w_0(x) , &&\text{ for } x\in \GG_0.
	\end{aligned}
	\right.
	\end{equation}
	The diffusion coefficients $\dO,\dG,\dGG$ are positive constants. The nonlinearities $f_1, f_2, f_3: \mathbb R^3 \to \mathbb R$ are locally Lipschitz continuous, and satisfy for all $(u,w,z)\in [0,\infty)^3$
	\begin{enumerate}[label=(A\theenumi),ref=A\theenumi] 
		\item\label{A1} $f_1(0,w,z) \geq 0$, $f_2(u,0,z) \geq 0$ and $f_3(u,w,0) \geq 0$;
		\item\label{A2} $f_2(u,w,z) + f_3(u,w,z) \leq 0$;
		\item\label{A3} $f_1(u,w,z) + f_2(u,w,z) \leq 0$ or $f_1(u,w,z) + f_3(u,w,z) \leq 0$;
		\item\label{A4_1} $f_1(u,w,z) \leq C(w^{\alpha} + z^{\alpha} + 1)$ for some $\alpha > 0$;
		\item\label{A4} $f_2(u,w,z) \leq C(w^{\beta}+z^{\beta}+1)$ or $f_3(u,w,z) \leq C(w^{\beta} + z^{\beta}+1)$ for some $\beta < \frac{d+4}{d+2}$; and
		\item\label{A5} $f_2, f_3$ have at most polynomial growth, i.e. there exists a three variables polynomial $P(a,b,c)$ such that $|f_2(a,b,c)|,  |f_3(a,b,c)| \leq P(a,b,c)$ for all $(a,b,c) \in [0,\infty)^3$.
	\end{enumerate}
	It is straightforward to verify that the nonlinearities in system \eqref{sys} satisfy the conditions \eqref{A1}--\eqref{A5} with $\alpha = 2$ and $\beta = 1$. The condition \eqref{A1} is usually called \textit{quasi-poisitivity} of the nonlinearities. It is essential in many modelling problems since it ensures the preservation of nonnegativity, which must be held for densities, concentrations, or populations. The assumptions \eqref{A2} and \eqref{A3} are referred to as \textit{mass dissipation}, which implies that some mass totals do not increase in time. \eqref{A4_1} and \eqref{A5} are quite general as no restrictions are imposed on $\alpha$ or the order of the polynomial $P$. The only technical condition we impose is \eqref{A4} requiring either $f_2$ or $f_3$ have upper bounds with \textit{sub-critical} growth rate.
	
	\medskip
	Before going into further discussion, we introduce some notation. For $p, q\geq 1$, we denote by
	\begin{equation*}
	L^p_{L^q(\Omega)}:= \left\{ u: [0,T]\to \bigcup_{t\in [0,T]}L^q(\O(t))\times \{t\}, \; t \mapsto (\hat{u}(t),t) \; \text{s.t.} \; \phi_{-(\cdot)}\hat{u}(\cdot) \in L^{p}(0,T;L^q(\O_0)) \right\}.
	\end{equation*}
	Here the mapping $\phi_{-t}$ is defined by $\phi_{-t}u:= u\circ \Phi_t$ which defines a linear isomorphism between $L^p(\O(t))$ and $L^p(\O_0)$. The spaces $L^p_{L^q(\Gamma)}$, $L^p_{W^{1,q}(\Omega)}$ and $L^p_{W^{1,q}(\GG)}$ for $p,q \geq 1$
	are defined similarly. The spaces $H^1_{H^1(\Omega)^*}$ and $H^1_{L^2(\Omega)}$ are defined as 
	\begin{equation*}
	H^1_{H^1(\Omega)^*} = \{u\in L^1_{H^1(\Omega)} \,|\, \dot{u} \in L^2_{H^1(\Omega)^*} \}, \quad H^1_{L^2(\Omega)} = \{u\in L^2_{L^2(\Omega)}\,|\, \dot{u} \in L^2_{L^2(\Omega)} \},
	\end{equation*}
	where the weak material derivative $\dot{u} \in L^2_{H^1(\Omega)^*}$ is understood as
	\begin{equation*}
	\int_{0}^T\langle \dot u(t), \eta(t)\rangle_{H^1(\Omega(t))^*, H^1(\Omega(t))} = -\int_0^T\intO u(t)\dot{\eta}(t) - \int_0^T\intO u(t)\eta(t)\na\cdot \VV_p.
	\end{equation*}
	The space $H^1_{L^2(\Gamma)}$ is defined similarly. For more properties of these function spaces, the interested reader is referred to \cite{alphonse2021function}.
	%
	
	\medskip
	The first main result of this paper is the following.
	\begin{theorem}[Global existence and boundedness of solutions]\label{thm:main}
		Assume \eqref{A1}--\eqref{A5} and \eqref{Phi1}--\eqref{Phi2}. Then for any nonnegative initial data $(u_0,w_0,z_0) \in L^\infty(\O_0)\times L^\infty(\GG_0)\times L^\infty(\GG_0)$, there exists a unique bounded nonnegative weak solution $(u,w,z)$ to \eqref{sys_extend} with 
		\begin{equation*}
		\|u\|_{L^{\infty}_{L^\infty(\O)}}, \|w\|_{L^{\infty}_{L^\infty(\GG)}}, \|z\|_{L^{\infty}_{L^\infty(\GG)}} \leq C(T).
		\end{equation*}
		By a weak solution we mean a triple $(u,w,z)\in H^1_{H^1(\Omega)^*}\cap L^2_{H^1(\Omega)}\times (H^1_{L^2(\Gamma)}\cap L^2_{H^1(\Gamma)})^2$ with $w\in L^{\infty}_{L^{\infty}(\Gamma)}$ satisfying for all test functions $(\eta, \psi, \xi)\in L^2_{H^1(\Omega)}\times (L^2_{H^1(\Gamma)})^2$,
		\begin{equation*}
		\begin{aligned}
		\langle \dot u, \eta \rangle + \intO u\eta \na\cdot \VV_p + \delta_{\Omega}\na u\cdot \na \eta + \intO\na\cdot(\J_{\Omega}u)\eta= \intG f_1(u,w,z)\eta + \intG ju\eta,\\
		\langle \dot w, \psi\rangle + \intG w\psi \naG\cdot \VV_p + \delta_\Gamma\intG \naG w\cdot \naG \psi + \intG \naG\cdot(\J_\Gamma w)\psi= \intG f_2(u,w,z)\psi,\\
		\langle \dot z, \xi \rangle + \intG z\xi \naG \VV_p+ \delta_{\Gamma'}\intG \naG z \cdot \naG \xi + \intG \naG\cdot(\J_\Gamma z)\xi=  \intG f_3(u,w,z)\xi,
		\end{aligned}
		\end{equation*}
		for almost every $t\in [0,T]$.
	\end{theorem}
	
	
	As \eqref{sys} is just a special case of \eqref{sys_extend}, we have
	\begin{corollary}[Global boundedness of solutions to \eqref{sys}]\label{cor:sys}
		Assume \eqref{Phi1}--\eqref{Phi2}. Then, for any nonnegative $(u_0, w_0, z_0) \in L^{\infty}(\Omega_0)\times L^\infty(\GG_0)\times L^\infty(\GG_0)$, there exists a unique global bounded nonnegative weak solution $(u,w,z)$ to \eqref{sys} with
		\begin{equation*}
		\|u\|_{L^\infty_{L^\infty(\O)}}, \|w\|_{L^\infty_{L^\infty(\GG)}}, \|z\|_{L^\infty_{L^\infty(\GG)}} \leq C(T),
		\end{equation*}
		which satisfy the conservation laws \eqref{law1} and \eqref{law2}.
	\end{corollary}
	
	\medskip
	If the domain is fixed, i.e. $\Omega(t)\equiv \Omega_0$, $\Gamma(t)\equiv \Gamma_0$, for all $t>0$, it was shown in \cite{alphonse2018coupled} by an  entropy method that the weak solution converges exponentially to the unique positive equilibrium. Due to the technique of proving boundedness of solutions, this convergence still requires the dimension restriction, i.e. $d\leq 2$. An immediate consequence of Corollary \ref{cor:sys} is that we can obtain this convergence without any dimension restriction. One natural question is that: \textit{can one get similar large time dynamics allowing the domain to change but material velocities decay to zero as $t\to\infty$?} We will answer this question affirmatively by adapting the entropy method to the case of evolving domains, under the assumptions that the volume of the domain and area of the boundary remain unchanged as $t\to\infty$. 
	
	To determine the large time behaviour of \eqref{sys} we consider, for given positive constants $M_1, M_2$, the following algebraic system
	\begin{equation}\label{equi_sys}
	\begin{cases}
	\zinf = \uinf\winf,\\
	\uinf \mathcal{L}^{d+1}(\Omega_0) + \zinf \mathcal{L}^{d}(\Gamma_0) = M_1,\\
	\winf \mathcal{L}^{d}(\Gamma_0) + \zinf \mathcal{L}^{d}(\Gamma_0) = M_2,
	\end{cases}
	\end{equation}
	where $\mathcal{L}^{d+1}$ and $\mathcal{L}^{d}$ are Lebesgue measures in $\R^{d+1}$ and $\R^{d}$ respectively. The existence and uniqueness of a positive solution $(\uinf,\zinf,\winf)$ to \eqref{equi_sys} follows from elementary calculations. Note that if all velocities are zero, i.e. $\VV_\Omega = \VV_\Gamma = \VV_p = 0$, or in other words, the domain is fixed $\Omega(t)\equiv \Omega_0$ and $\Gamma(t)\equiv \Gamma_0$ for all $t>0$, then the triple $(\uinf, \zinf, \winf)$ is the equilibrium of the system \eqref{sys} as it was considered in \cite{alphonse2018coupled}.
	The large time behaviour of solutions to \eqref{sys} is stated in the following theorem.
	\begin{theorem}\label{thm:convergence}
		Assume the domains satisfy
		\begin{equation}\label{B}\tag{B}
		\mathcal{L}^{d+1}(\Omega(t)) = \mathcal{L}^{d+1}(\Omega_0) \quad \text{ and } \quad \mathcal{L}^{d}(\Gamma(t)) = \mathcal{L}^d(\Gamma_0), \quad\text{for all } t> 0
		\end{equation}
		and the following inequalities hold 
		\begin{equation}\label{aLSI}
		\intO \frac{|\na f|^2}{f} \ge C_{\text{LSI},\Omega}\intO f\log \frac{f}{\overline f}, \quad \intG \frac{|\na_\GG g|^2}{g} \ge C_{\text{LSI},\GG}\intG g\log\frac{g}{\overline g},
		\end{equation}	
		\begin{equation}\label{aTrPW}
		\intO|\na f|^2 \ge C_{\text{TrPW}}\intG |f - \overline f|^2, \quad \intG |\na_\GG g|^2 \ge C_{\text{PW}}\intG |g - \overline g|^2
		\end{equation}
		with $\overline f = \measv(\O(t))^{-1}\intO f(x)$ and $\overline g = \measb(\GG(t))^{-1}\intG g(x)$, where the constants $C_{LSI,\Omega}$, $C_{LSI,\GG}$, $C_{TrPW}$, $C_{PW}$ can be chosen independent of $t>0$.
		
		\medskip
		If the velocities satisfy
		\begin{equation}\label{decay}
		\limsup_{t\to\infty}\left(\|\VV_\O(t)\|_{L^\infty(\R^d)} + \|\VV_\GG(t)\|_{L^\infty(\R^d)}\right) = 0
		\end{equation}
		then
		\begin{equation*}
		\lim_{t\to\infty}(\|u(t) - \uinf\|_{L^1(\O(t))} + \|w(t)-\winf\|_{L^1(\GG(t))} + \|z(t) - \zinf\|_{L^1(\GG(t))}) = 0.
		\end{equation*}
		Moreover, if
		\begin{equation}\label{ex_decay}
		\|\VV_\O(t)\|_{L^\infty(\R^d)} + \|\VV_\GG(t)\|_{L^\infty(\R^d)} \leq Ce^{-\delta t}
		\end{equation}
		for some $\delta >0$, then there exists $\delta>\mu >0$ such that 
		\[
		\|u(t) - \uinf\|_{L^1(\O(t))} + \|w(t)-\winf\|_{L^1(\GG(t))} + \|z(t) - \zinf\|_{L^1(\GG(t))} \leq Ce^{-\mu t}.
		\]
	\end{theorem}
	\begin{remark}\hfill
		\begin{itemize}
			\item Although $(\uinf, \zinf, \winf)$ is not an equilibrium of the system \eqref{sys}, we refer the results in Theorem \ref{thm:convergence} as {\normalfont convergence to equilibrium} since it is the extension of that in \cite{alphonse2018coupled} or other works in the literature.
			\item The assumption \eqref{B} is needed for us to be able to prove the functional inequality which is essential in the entropy method. It can be obtained, for instance if the velocities satisfy $\VV_\O \cdot \nu = 0$ and $\naG\cdot \VV_\Gamma = 0$. We believe that the convergence to equilibrium still holds if the volume $\Omega(t)$ and area $\Gamma(t)$ converge to fixed values as $t\to\infty$. Investigating this issue might require some stability results of vector-valued functional inequalities which goes beyond the scope of this current paper.
			\item It is remarked that the proof the convergence of solutions does not use the convergence of domains.
			\item We expect that the uniformity of constants for the Log-Sobolev-, Trace-Poincar\'e-, and Poincar\'e inequalities in \eqref{aLSI} and \eqref{aTrPW} can be proved by using \eqref{B} and \eqref{decay}, see e.g. \cite{boulkhemair2007uniform} for the case of Poincar\'e inequality. However, this goes beyond the scope of the present paper and therefore will be investigated in a separated work.
		\end{itemize}
	\end{remark}
	%
	
	\subsection{Key ideas}
	
	Let us sketch the main ideas used to prove the results in Theorems \ref{thm:main} and \ref{thm:convergence}.
	
	\medskip
	For Theorem \ref{thm:main}, the assumption \eqref{A1} gives nonnegativity of solutions for  nonnegative   initial data. From \eqref{A2} and \eqref{A3} we get
	\begin{equation}\label{Linfty-L1}
	\|u\|_{L^\infty_{L^1(\O)}}, \|w\|_{L^\infty_{L^1(\GG)}}, \|z\|_{L^\infty_{L^1(\GG)}} \leq M
	\end{equation}
	where $M$ is a constant independent of $T$. Concerning assumption \eqref{A4} we consider only $f_2(u,w,z) \leq C(w^{\beta_1} + z^{\beta_2} + 1)$, since the case $f_3(u,w,z) \leq C(w^{\beta_1} + z^{\beta_2}+1)$ can be treated similarly.  Taking the sum of the equations for $w$ and $z$, we get
	\begin{equation}\label{sum_wz}
	\dot{w} + \dot{z} + (w+z)\naG\cdot \VV_p - \Delta_\GG(\dG w + \dGG z) + \naG\cdot (\J_\GG(w+z)) \leq 0.
	\end{equation}
	By a duality lemma (Lemma \ref{duality-L2}) we get
	\begin{equation*}
	\|w\|_{L^2_{L^2(\GG)}} + \|z\|_{L^2_{L^2(\GG)}} \leq C(T).
	\end{equation*}In the assumption \eqref{A4} we have $f_2(u,w,z) \leq C(w^{\beta} + z^{\beta} + 1)$, and since $\beta < \frac{d+4}{d+2} < 2$, it follows that
	\begin{equation*}
	\dot{w} + w\naG\cdot\VV_p - \dG\Delta_\GG w + \naG\cdot(\J_\GG w) \leq C(w^{\beta} + z^{\beta} + 1) \in L^{s_0}_{L^{s_0}(\GG)}
	\end{equation*}
	for some $s_0 >1$. 	We obtain by the heat regularity (Lemma \ref{heat-regularity}) and the comparison principle (Lemma \ref{comparison-surface}) that $w \in L^{s_1}_{L^{s_1}(\GG)}$ with some $s_1 > s_0$. By this and another duality lemma (Lemma \ref{duality-1}) we see that $z \in L^{s_1}_{L^{s_1}(\GG)}$.
	Repeating this procedure gives a sequence $\{s_n\}$ with $\lim_{n\to\infty}s_n = \infty$, therefore $w, z\in L^r_{L^r(\GG)}$ for all $r\in [1,\infty)$. 
	Now from \eqref{A4_1}
	\begin{equation*}
	\begin{aligned}
	\dot{u} + u\na\cdot \VV_p - \dO \Delta u + \na\cdot (\J_\O u) &= 0, &&\text{ in } \O(t),\\
	\dO \na u \cdot \nu - uj &\leq C(w^{\alpha} + z^{\alpha} + 1), && \text{ on } \GG(t),\\
	\end{aligned}
	\end{equation*}
	it follows that $u \in L^{\infty}_{L^\infty(\O)}$ thanks to the comparison principle (Lemma \ref{comparison-Neumann}) and regularity of the heat equation with Neumann boundary condition (Lemma \ref{L-infinity-Neumann}). In addition $u \in L^2_{H^1(\Omega)}$. Therefore, the trace inequality implies $u\in L^r_{L^r(\GG)}$ for all $r\in [1,\infty)$. Now, thanks to \eqref{A5}, one finally concludes $w, z\in L^\infty_{L^\infty(\GG)}$ (by Lemma \ref{L-infinity-manifold}).

	\medskip
	
	To prove the convergence to equilibrium result in Theorem \ref{thm:convergence} we apply the so-called entropy method. Firstly, thanks to assumption \eqref{B}, we can define a unique equilibrium satisfying \eqref{equi_sys}. Next, consider the relative entropy
	\begin{align*}
	E[u,w,z] = \intO \left(u\log\frac{u}{\uinf} - u + \uinf\right)  &+ \intG \left(w\log\frac{w}{\winf} - w + \winf \right) \\&+ \intG \left(z\log\frac{z}{\zinf} - z + \zinf \right) .
	\end{align*}
	Unlike the case of fixed domains, $E[u,w,z]$ here is not exactly a Lyapunov function since there are correction terms due to the fact that the domains are also time-dependent. Nevertheless, we can estimate
	\begin{equation*}
	D[u,w,z]:= -\frac{d}{dt}E[u,w,z] \geq \wt{D}[u,w,z] - Ch(t)
	\end{equation*}
	where $h(t)$ is a function decaying to zero as $t\to\infty$, and $\wt{D}[u,w,z]$ is the usual entropy dissipation (as in the fixed domain case) defined by
	\begin{equation*}
	\wt{D}[u,w,z] = \frac \dO 2\intO\frac{|\na u|^2}{u}  + \frac\dG 2\intG\frac{|\naG w|^2}{w}  + \frac\dGG 2\intG\frac{|\naG z|^2}{z}  + \intG(z-uw)\log\frac{z}{uw} .
	\end{equation*}	
	The cornerstone of the entropy method is to prove the functional inequality
	\begin{equation}\label{e1}
	\boxed{\wt{D}[u,w,z] \geq \lambda E[u,w,z]}
	\end{equation}
	under the mass conservation laws \eqref{law1} and \eqref{law2}. We remark that \textbf{\eqref{e1} is a pure vector-valued functional inequality} in the sense that it does not use any property of system \eqref{sys}. Because of this, we believe that it can be useful in other situations. This inequality is proved by adapting ideas in fixed domains, see e.g. \cite{fellner2018well,fellner2018convergence}, thanks to the fact that the domain volume and surface area are fixed and the fact that the constants in \eqref{aLSI} and \eqref{aTrPW} are not depending on time. With this inequality at hand, we arrive at
	\begin{equation*}
	\frac{d}{dt}E[u,w,z] \leq -\lambda E[u,w,z] + Ch(t).
	\end{equation*}
	By Gronwall's inequality and the fact that $h(t)$ is decaying, we obtain 
	\begin{equation*}
	\lim_{t\to\infty}E[u,w,z](t) = 0.
	\end{equation*}
	And moreover, under the assumption \eqref{ex_decay},
	\begin{equation*}
	E[u,w,z](t) \leq Ce^{-\mu t}
	\end{equation*}
	for some $\mu>0$. By the Csisz\'ar-Kullback-Pinsker inequality (Lemma \ref{CKP}), we obtain the desired decay of the solution in $L^1$-norm.
	
	%
	%
	\medskip
	\noindent{\bf Notation.}
	\begin{itemize}
		\item For any $T>0$ we write $Q_T = \bigcup_{t\in (0,T)}\O(t)\times \{t\}$ and $\GGG_T = \bigcup_{t\in(0,T)}\GG(t)\times \{t\}$.
		\item We denote by $C$ a generic nonnegative constant , whose value can be different from line to line (or even in the same line).
		\item For simplicity we write $\|\VV_\O\|_\infty$ and $\|\VV_\GG\|_\infty$ for $\|\VV_\O\|_{L^\infty_{L^\infty(\O)}}$ and $\|\VV_\GG\|_{L^\infty_{L^\infty(\GG)}}$ respectively.
		\item We will always omit the differentials in the integrals, i.e., we write $\int_0^T f$, $\int_{\Omega(t)} f$ or $\int_{\Gamma(t)} f$ rather than $\int_0^T f \dt$, $\int_{\Omega(t)} f \dx$ or $\int_{\Gamma(t)} f \mathrm d\sigma$.
	\end{itemize}

	
	\section{Global existence and boundedness}
	This section aims at proving Theorem \ref{thm:main}. The proof combines duality methods in evolving surfaces and heat regularisation. To obtain the duality estimates, we first show in Subsection \ref{subs:maximal_regularity} an $L^p$-maximal regularity in moving surfaces. This result will be used in Subsection \ref{subs:duality} to prove two important duality estimates, which are the starting point of proving Theorem \ref{thm:main}. We would like to emphasise that these two subsections are of independent interest, and, to our belief, could be used in many other situations. The heat regularisation in moving domains is shown in Subsection \ref{subs:heat_regularisation}, and the proof of Theorem \ref{thm:main} is finally presented in the last subsection.
	\subsection{$L^p$-maximal regularity in evolving surfaces}\label{subs:maximal_regularity}

	\begin{lemma}\label{max-regular}
		Let $m>0$, and $y$ be the solution to the parabolic equation
		\begin{equation}\label{evol}
		\begin{aligned}
		\dot y - m \Delta_\GG y + \naG y \cdot \VV_\GG &= f(x,t), \quad \text{ for } x\in \Gamma(t),\\
		y(x,0) &= 0, \quad \text{ for } x\in \Gamma_0.
		\end{aligned}
		\end{equation}
		Then, for each $1< p <\infty$ and $f\in L^p_{L^p(\GG)}$, there exists a constant $C_{m,p,T}>0$ such that
		\begin{equation*}
		\|\dot y \|_{L^p_{L^p(\Gamma)}} + \sum_{s=0}^2\|\pa_{x,\Gamma}^s y\|_{L^p_{L^p(\Gamma)}} \leq C_{m,p,T}\|f\|_{L^p_{L^p(\GG)}}
		\end{equation*}
		In particular, 
		\begin{equation*}
		\|y\|_{L^p_{L^p(\GG)}} + \|\dot y\|_{L^p_{L^p(\GG)}} + \|\Delta_\GG y\|_{L^p_{L^p(\GG)}} \leq C_{m,p,T}\|f\|_{L^p_{L^p(\GG)}}.
		\end{equation*}
	\end{lemma}
	\begin{remark}\label{backward}
		By changing variable $t$ to $T-t$, the maximal regularity in Lemma \ref{max-regular} also applies to the backward equation
		\begin{equation*}
		\begin{aligned}
		\dot y + m\Delta_\GG y + \naG y\cdot \VV_\GG &= f(x,t), \quad \text{ on } \GG(t),\\
		y(T) &= 0.
		\end{aligned}
		\end{equation*}
	\end{remark}
	
	The idea of the proof is to transform the equation on the moving interfaces to a fixed reference surface, and then do a further change of variables onto an open domain in $\R^d$ where the $L^p$-maximal regularity results are known. To do so, we make use of the calculations in \cite{elliott2015time}, as well as the ideas of \cite[Theorem 3.5]{sharma2016global}. First we introduce some notation. For a bounded domain $B\subset \mathbb R^{n-1}$, we denote 
	\begin{equation*}
	W^{2,1}_p(B\times(0,T)):= \left\{f\in L^p(B\times(0,T)) \;\big| \; \pa_t^r\pa_x^sf\in L^p(B\times(0,T)) \text{ for } r,s\in \mathbb N, 2r + s \leq 2 \right\}
	\end{equation*}
	equipped with the norm
	\begin{equation*}
	\|f\|_{W^{2,1}_p(B\times(0,T))}:= \|\pa_t f\|_{L^p(B\times (0,T))} + \sum_{s=0}^2\|\pa_x^sf\|_{L^p(B\times(0,T))}.
	\end{equation*}		
	For a function $u: \Gamma_0 \to \R$, the {\it tangential gradient} at $x\in \Gamma_0$ is defined as
	\begin{equation*}
	\na_\Gamma u(x) = \na \hat{u}(x) - (\nu(x)\cdot \na \hat{u}(x))\nu(x) = P(x)\na \hat{u}(x),
	\end{equation*}
	where $\hat{u}$ denotes an extension of $u$ to an open neighbourhood of $x$ in $\R^{n-1}$ and $P(x): \R^{n-1}\to T_x\Gamma_0$ denotes the projection onto the tangent space $T_x\Gamma_0$, defined by
	\begin{equation*}
	P(x) = \text{Id} - \nu(x)\otimes \nu(x), \quad \text{ with components } P_{ij}(x) = \delta_{ij}x + \nu_i(x)\nu_j(x).
	\end{equation*}
	We denote $\na_{\Gamma} = (\uD_1, \ldots, \uD_{n-1})$. Define the {\it Weingarten map} on $\Gamma_0$ as $H(x) = \na_{\Gamma}\nu(x) \in \R^{(n-1)\times(n-1)}$, and we have the commutator rule
	\begin{equation*}
	\uD_i\uD_j u - \uD_j\uD_i u = (H_{jk}\nu_i - H_{ik}\nu_j)\uD_k u, \quad 1\leq i,j,k\leq n-1,
	\end{equation*}
	where hereafter we will use the Einstein summation convention: 
	
	\medskip
	We have the family of diffeomorphisms $\Phi_t: \Gamma_0 \to \Gamma(t)$, and denote
	\begin{equation*}
	G: \bar{\Gamma}_0\times [0,T]\to \R^{(n-1)\times(n-1)}, \quad G_{ij}(x,t) = \uD_i\Phi_t(x)\uD_j\Phi_t(x) + \nu_i(x)\nu_j(x), \; 1\le i,j \le n-1,
	\end{equation*}
	and the inverse $G^{-1} = (G^{ij})_{i,j}$. The matrix $G$ is invertible and positive semi-definite at all points, thus it is positive-definite. Since it is also symmetric, the same is true for $G^{-1}$.
	
	\begin{proof}
		
		\noindent\underline{Step 1. Transform \eqref{evol} to an equation on $\Gamma_0$.}
		For a solution $y$ of \eqref{evol}, we denote the pullbacks to the reference surface $\Gamma_0$ as
		\begin{align*}
		\wt{y}(x,t):= y(\Phi_t(x),t), \qquad x\in\Gamma_0, t\in [0,T],\\
		\wt{f}(x,t):= f(\Phi_t(x), t), \qquad x\in\Gamma_0, t\in [0,T].
		\end{align*}
		As in \cite{elliott2015time}, it follows that $\wt{y}$ solves the system
		\begin{equation}\label{ee1}
		\begin{aligned}
		\pa_t \wt{y} - m\uD_i(G^{ij}\uD_j \wt{y}) - \frac{m}{2}P_{ij}G^{jk}G^{ml}\uD_mG_{ik}\uD_l \wt{y} = \wt{f}, &
		\quad (x,t)\in \Gamma_0\times (0,T),\\
		\wt{y}(x,0) = 0, &\quad x\in\Gamma_0. 
		\end{aligned}
		\end{equation}
		
		\medskip
		\noindent\underline{Step 2. Locally transform \eqref{ee1} to an equation on open sets of $\mathbb R^{n-1}$.}
		For a point $x_0 \in \Gamma_0$ and $(U_{x_0}, \varphi)$ a local chart at $x_0$, meaning that $U_{x_0}\subset \Gamma_0$ is an open set containing $x_0$ and $\varphi: \R^{n-1}\to U_{x_0}$ is a diffeomorphism such that $\varphi(0) = x_0$ and $\varphi(B_2(0)) \subset U_{x_0}$. Define
		\begin{equation*}
		\wh{y}(z,t) = \wt{y}(\varphi(z),t), \quad (z,t)\in B_2(0)\times (0,T).
		\end{equation*}
		Defining, for $z\in B_2(0)$, $t\in (0,T)$ and $1\leq i,j\leq n-1$,
		\begin{equation*}
		\wh{f}(z,t) = \wt{f}(\varphi(z),t), \quad \wh{G}_{ij}(z) = G_{ij}(\varphi(z)), \quad \wh{G}^{ij}(z) = G^{ij}(\varphi(z)), \quad \wh{P}_{ij}(z) = P_{ij}(\varphi(z))
		\end{equation*}
		we have
		\begin{equation*}
		\pa_t \wh{y}(z,t) = \wh{f}(z,t) + m(\uD_i(G^{ij}\uD_j \wt{y}))(\varphi(z)) + \frac{m}{2}\wh{P}_{ij}\wh{G}^{jk}\wh{G}^{ml}(\uD_mG_{ik})(\varphi(z))(\uD_l\wt{y})(\varphi(z)).
		\end{equation*}
		Using \cite[Equation (2.11)]{elliott2015time} we have the following: if $F: \Gamma_0 \to \R$ and $\varphi: B_2(0) \to \Gamma_0$, then for $\wh{F} = F\circ \varphi: B_2(0) \to \R$ we have
		\begin{equation*}
		\na \wh{F}(z) = D\varphi(z)(\na_\Gamma F)(\varphi(z)) \quad \Longrightarrow \quad (\na_\Gamma F)(\varphi(z)) = (D\varphi(z))^{-1}\na \wh{F}(z),
		\end{equation*}
		and writing $c^{ij} = (D\varphi)_{ij}^{-1}$ this is
		\begin{equation*}
		(\uD_iF)(\varphi(z)) = c^{ij}(z)D_j\wh{F}(z).
		\end{equation*}
		Using the above then allows us to write
		\begin{equation*}
		\begin{aligned}	
		(\uD_i(G^{ij}\uD_j\wt{y}))(\varphi(z)) &= c^{ik}(z)D_k(\wh{G}^{ij}(z)c^{jp}(z)D_p\wh{y}),\\
		(\uD_mG_{ik})(\varphi(z))(\uD_l\wt{y})(\varphi(z)) &= c^{mr}(z)c^{ls}(z)D_r(\wh{G}_{ik}(z))(D_s\wh{y}(z)).
		\end{aligned}	
		\end{equation*}
		For convenience we rewrite the first term as
		\begin{equation*}
		(\uD_i(G^{ij}(\uD_j\wt{y})))(\varphi(z)) = D_k(c^{ik}\wh{G}^{ij}c^{jp}D_p\wh{y}) - D_kc^{ik}\wh{G}^{ij}c^{jp}D_p\wh{y},
		\end{equation*}
		from where we conclude that $\wh{y}$ satisfies the equation
		\begin{equation}\label{ee2}
		\pa_t \wh{y} - mD_k(c^{ik}\wh{G}^{ij}c^{jp}D_p\wh{y}) + \left(D_kc^{ik}\wh{G}^{ij}c^{js} - \frac{m}{2}\wh{P}_{ij}\wh{G}^{jk}\wh{G}^{ml}c^{mr}c^{ls}D_r\wh{G}_{ik} \right)D_s\wh{y} = \wh{f}(z,t).
		\end{equation}
		The above can be rewritten as
		\begin{equation*}
		\pa_t \wh{y}(z,t) - m\na\cdot\left(\wh{A}(z,t)\na \wh{y}(z,t) \right) + \wh{b}(z,t)\cdot \na y(z,t) = \wh{f}(z,t),
		\end{equation*}
		where the matrix $\wh{A}(z,t)$ and the vector $\wh{b}(z,t)$ have entries
		\begin{equation*}
		\wh{A}_{ij}(z,t) = c^{ki}\wh{G}^{kl}c^{lj} \quad \text{ and } \quad \wh{b}_i(z,t) = D_kc^{lk}\wh{G}^{lj}c^{ji} - \frac{m}{2}\wh{P}_{sj}\wh{G}^{jk}\wh{G}^{ml}c^{mr}c^{li}D_r\wh{G}_{sk},
		\end{equation*}
		or more precisely
		\begin{equation*}
		\wh{A}_{ij}(z,t) = \sum_{k,l=1}^{n-1}c^{ki}\wh{G}^{kl}c^{lj}
		\end{equation*}
		and
		\begin{equation*}
		\wh{b}_i(z,t) = \sum_{j,k,l=1}^{n-1}D_kc^{lk}\wh{G}^{lj}c^{ji} - \frac m2 \sum_{j,k,l,m,r,s=1}^{n-1}\wh{P}_{sj}\wh{G}^{jk}\wh{G}^{ml}c^{mr}c^{li}D_r\wh{G}_{sk}.
		\end{equation*}
		Due to the assumptions on the evolution of the surfaces, it is clear that both $\wh{A}$ and $\wh{b}$ are bounded. To check that $\wh{A}$ is coercive, note that $\wh{G}$ is positive-definite (because $G$ is), and so, for $z = (z_1, \ldots, z_{n-1})$ and defining
		\begin{equation*}
		v = D\varphi^{-1}z \quad \text{ with coordinates } \quad v_k = c^{ki}z_i
		\end{equation*}
		we have
		\begin{equation*}
		\wh{A}(z,t)z\cdot z = c^{ki}\wh{G}^{kl}c^{lj}z_jz_i = \wh{G}v\cdot v \geq \lambda_{\wh{G}}|v|^2
		\end{equation*}
		for some $\lambda_{\wh{G}}>0$. 	Since $\varphi$ is diffeomorphism, there exists $C_\varphi>0$ such that $\|D\varphi\| \le C$, from where we obtain
		\begin{equation*}
		|z| = |D\varphi v| \leq \|D\varphi\|\,|v| \leq C_\varphi|v| \quad |v| \ge \frac{1}{C_\varphi}|z|,
		\end{equation*}
		and therefore
		\begin{equation*}
		\wh{A}(z,t)z\cdot z \geq \lambda_{\wh{G}}|v|^2 \geq \frac{\lambda_{\wh G}}{C_\varphi^2}|z|^2,
		\end{equation*}
		as desired. Note that we have not given boundary data for the equation \eqref{ee2} of $\wh{y}$. We proceed further by using a smooth cut-off function $\xi: B^{n-1}_2(0) \to [0,1]$ such that $\xi(z)\equiv 1$ for $|z| \leq 1$ and $\xi(z) \equiv 0$ for $|z| \geq 3/2$. Define $\bb{y}(z,t):= \wh{y}(z,t)\xi(z)$. It's easy to see that $\bb{y}$ solves the following equation with homogeneous Dirichlet boundary condition on $B^{n-1}_2(0)$,
		\begin{equation}\label{equation_why}
		\begin{cases}	
		\pa_t \bb{y} = m\na\cdot[\wh A \na \bb{y}] + \wh b \cdot \na \bb{y} + \bb g(z,t), &z\in B^{n-1}_2(0), \; t\in(0,T),\\
		\bb{y}(z,t) = 0, &z\in\pa B^{n-1}_2(0), \, t\in (0,T),\\
		\bb{y}(z,0) = 0, &z\in B^{n-1}_2(0)
		\end{cases}
		\end{equation}
		where
		\begin{align*}
		\bb{g}(z,t) &= -2m[\na \wh{y}(z,t)]^{\top}\wh{A}(z,t)\na \xi(z)  - m\wh{y}(z,t)\na\cdot[\wh{A}(z,t)\na \xi(z)]\\
		&\qquad - \wh{y}(z,t)\wh{b}(z,t)\cdot \na \xi(z) + \xi(z)\wh{f}(z,t).
		\end{align*}
		Note that the ``external term" $\bb{g}$ depends on $\wh y$ and $\na \wh y$.  Denote by $Q_T:= B^{n-1}_2(0)\times (0,T)$. Thanks to properties of $\wh{A}$ and $\wh{b}$ and the fact that $\xi$ is smooth, we have
		\begin{equation*}
		\|\bb{g}\|_{L^p(Q_T)}\leq C_T\left(\|\wh{y}\|_{L^p(Q_T)} + \|\na \wh{y}\|_{L^p(Q_T)} + \|\wh{f}\|_{L^p(Q_T)} \right).
		\end{equation*}
		We will use the interpolation inequality: for any $\epsilon > 0$, there exists $C_\epsilon>0$ such that
		\begin{equation}\label{interpolation}
		\|\na \wh{y}\|_{L^p(B^{n-1}_2(0))} \leq \epsilon\|\wh y\|_{L^p(B^{n-1}_2(0))} + C_\epsilon\|\na^2\wh{y}\|_{L^p(B^{n-1}_2(0))}.
		\end{equation}
		Indeed, by Gagliardo-Nirenberg's inequality
		\begin{equation*}
		\|\na \wh y\|_{L^p(B^{n-1}_2(0))} \leq C\|\na^2 \wh y\|_{L^p(B^{n-1}_2(0))}^{1/2}\|\wh y\|_{L^p(B^{n-1}_2(0))}^{1/2} + C\|\wh y\|_{L^p(B^{n-1}_2(0))}.
		\end{equation*}
		An application of the Young inequality gives immediately \eqref{interpolation}. It then follows that
		\begin{equation*}
		\|\bb{g}\|_{L^p(Q_T)}\leq C_T\left(\|\wh{y}\|_{L^p(Q_T)} + \epsilon\|\na^2 \wh{y}\|_{L^p(Q_T)} + \|\wh{f}\|_{L^p(Q_T)} \right).
		\end{equation*}
		We now apply the $L^p$-maximal regularity to the equation \eqref{equation_why}, see e.g. \cite[Theorem IV.9.1]{LSU68}, to get
		\begin{equation}\label{h1}
		\|\bb{y}\|_{W^{2,1}_p(Q_T)} \leq C_T\|\bb{g}\|_{L^p(Q_T)} \leq C_T(\|\wh{y}\|_{L^p(Q_T)} + \|\wh{f}\|_{L^p(Q_T)}) + \varepsilon \|\na^2 \wh{y}\|_{L^p(Q_T)}.
		\end{equation}

		\medskip
		\noindent\underline{Step 3. Going back to $\Gamma_0\times (0,T)$ and gluing.} We have $\eta = \wh{y}$ on $B_1(0)$, and thus from \eqref{A5} we immediately obtain
		\begin{equation}\label{ee3}
		\|\wh{y}\|_{W_p^{2,1}(B_1(0)\times (0,T))} \leq \|\bb{y}\|_{W^{2,1}_p(Q_T)} \leq C_T\left(\|\wh y\|_{L^p(Q_T)} + \|\wh{f}\|_{L^p(Q_T)} \right) + \varepsilon\|\na^2 \wh{y}\|_{L^p(Q_T)}.
		\end{equation}
		Now define $\wt{U}_{x_0} = \varphi(B_1(0))\subset U_{x_0}$ so that we have the bijection $\varphi: B_1(0)\to \wt{U}_{x_0}$. We reuse the notation $c^{ij} = (D\varphi)_{ij}^{-1}$. Recall that $\wh y = \wt y \circ \varphi$ and $\wh{f} = \wt f\circ \varphi$.
		
		\medskip
		\noindent\textbf{Estimate the left hand side of \eqref{ee3}}: For the terms on the left hand side, we note that, for the function $\wt{y}$, we have
		\begin{equation*}
		\|\wt y\|_{L^p(\wt U_{x_0}\times(0,T)}^p = \int_0^T\int_{B_1(0)}|\wh y|^p|\det D\varphi| \leq C_\varphi\|\wh y\|_{L^p(B_1(0)\times(0,T))}^p,
		\end{equation*}
		for the first spatial derivatives we have
		\begin{align*}
		\|\uD_i \wt{y}\|_{L^p(\wt U_{x_0}\times(0,T))}^p &= \int_0^T\int_{B_1(0)}|c^{ij}D_j\wh y|^d|\det D\varphi|\\
		&\leq C_{\varphi,c}\sum_{j=1}^{n}\int_0^T\int_{B_1(0)}|D_j\wh y|^p \leq C_{\varphi,c,n}\|\na \wh y\|_{L^p(B_1(0)\times (0,T))}^p
		\end{align*}
		and for the second derivatives, recalling
		\begin{equation}\label{ee4}
		(\uD_j\uD_i)\circ \varphi = c^{jk}D_k(\uD_i\wt y \circ \varphi) = c^{jk}D_k(c^{il}D_l\wh y) = c^{jk}D_kc^{il}D_l\wh{y} + c^{jk}c^{il}D_kD_l\wh{y},
		\end{equation}
		we get
		\begin{align*}
		\|\uD_j\uD_i \wt y\|_{L^p(\wt U_{x_0}\times (0,T))}^p &= \int_0^T\int_{B_1(0)}\left(|c^{jk}D_kc^{il}D_l\wh y|^p + |c^{jk}c^{il}D_kD_l\wh y|^p \right)|\det D\varphi|\\
		&\le C_{\varphi,c,n}\left(\|\na \wh y\|_{L^p(B_1(0)\times(0,T))}^p + \|\na^2\wh y\|_{L^p(B_1(0)\times(0,T))}^p \right).
		\end{align*}
		Note that $\wt y$ solves \eqref{ee1}, it follows from the above that 
		\begin{align*}
		\|\pa_t \wt y\|_{L^p(\wt U_{x_0}\times(0,T))} &\leq \|\wh f\|_{L^p(B_1(0)\times(0,T))} + C_{\Phi,G}\|\wt y\|_{W_p^{2,1}(\wt U_{x_0}\times(0,T))}\\
		& \leq \|\wh f\|_{L^p(B_1(0)\times(0,T))} +C_{\Phi,G,\varphi,c,n}\|\wh y\|_{W^{2,1}_p(B_1(0)\times(0,T))},
		\end{align*}
		and so we have, from \eqref{ee3}, 
		\begin{equation}\label{ee5}
		\begin{aligned}
		\|\wt y\|_{W_p^{2,1}(\wt U_{x_0}\times(0,T))} &\leq \leq \|\wh f\|_{L^p(B_1(0)\times(0,T))} + C_{\Phi,G,\varphi,c,n}\|\wh y\|_{W^{2,1}_p(B_1(0)\times(0,T))}\\
		&\leq \|\wh f\|_{L^p(Q_T)} + C_{\Phi,G,\varphi,c,n}\|\wh y\|_{W^{2,1}_p(Q_T)}\\
		&\leq C_{T,\Phi,G,\varphi,c,n}\left(\|\wh y\|_{L^p(Q_T)} + \|\wh f\|_{L^p(Q_T)} \right) + \varepsilon'\|\na^2 \wh{y}\|_{L^p(Q_T)}.
		\end{aligned}
		\end{equation}

		\medskip
		\noindent\textbf{Estimate the right hand side \eqref{ee3}:} We have
		\begin{equation}\label{ee6}
		\|\wh y\|_{L^p(Q_T)}^p = \int_0^T\int_{U_{x_0}}|\wt y|^p|\det D\varphi^{-1}| \leq C_{\varphi}\|\wt y\|_{L^p(U_{x_0}\times(0,T))}^p,
		\end{equation}
		and similarly
		\begin{equation}\label{ee7}
		\|\wh f\|_{L^p(Q_T)}^p \leq C_{\varphi}\|\wt f\|_{L^p(U_{x_0}\times(0,T))}^p.
		\end{equation}
		It follows from \eqref{ee4} that
		\begin{equation*}
		D_kD_l\wh{y} = c_{jk}c_{il}(\uD_j\uD_i\wt y)\circ \varphi - c_{il}D_kc^{il}D_l\wh{y}
		\end{equation*}
		which then gives by integration
		\begin{equation*}
		\|D_kD_l\wh y\|_{L^p(Q_T)} \leq C_{\varphi,c,n}\left(\|\na_{\Gamma}^2\wt y\|_{L^p(\wt U_{x_0}\times(0,T))} + \|\na \wh y\|_{L^p(Q_T)}\right)
		\end{equation*}
		implying
		\begin{equation}\label{ee8}
		\|\na^2 \wh y\|_{L^p(Q_T)} \leq C_{\varphi,c,n}\left(\|\na^2_{\Gamma}\wt y\|_{L^p(\wt U_{x_0}\times(0,T))} + \|\na_{\Gamma}\wt y\|_{L^p(\wt U_{x_0}\times(0,T))} \right).
		\end{equation}
		
		\medskip
		\noindent Combining \eqref{ee5}--\eqref{ee8} leads  to 
		\begin{equation}\label{ee8_1}
		\begin{aligned}
		\|\wt y\|_{W_p^{2,1}(\wt U_{x_0}\times(0,T))} &\leq C\left(\|\wt y\|_{L^p(U_{x_0}\times (0,T))} + \|\wt f\|_{L^p(U_{x_0}\times(0,T))} \right)\\
		&\quad + \varepsilon'\|\na \wt y\|_{L^p(U_{x_0}\times(0,T))} + \varepsilon'\|\na^2 \wt y\|_{L^p(U_{x_0}\times(0,T))},
		\end{aligned}
		\end{equation}
		where both $C$ and $\varepsilon'$ depend on $T, \Phi, G, \varphi, c, n$.
		
		\medskip
		\noindent{\underline{Step 4: Push-forward estimates to the moving surface.}} We start by analysing the right hand side of \eqref{ee8_1}. We have
		\begin{equation*}
		\|\wt y\|_{L^p(\Gamma_0\times(0,T))}^p = \int_0^T\int_{\Gamma_0}|y(\Phi_t(x),t)|^p = \int_0^T\int_{\Gamma(t)}|y(x,t)|^p|\det D\Phi_t^{-1}| \leq C_{\Phi}\|y\|_{L_{L^p(\Gamma)}^p}^p,
		\end{equation*}
		and the same estimate holds for $f$, from where we obtain
		\begin{equation}\label{ee9}
		\|\wt y\|_{W_p^{2,1}(\Gamma_0\times(0,T)} \leq C_T(C_{\Phi})^{1/p}\left(\|y\|_{L_{L^p(\Gamma)}^p} + \|f\|_{L_{L^p(\Gamma)}^p} \right).
		\end{equation}
		We now study the quantity $\|\dot y\|_{L^p_{L^p(\Gamma)}} + \sum_{s=0}^2 \|\pa_{x,\Gamma}^sy\|_{L^p_{L^p(\Gamma)}}$.
		\begin{itemize}
			\item Case $s = 0$. We have
			\begin{equation*}
			\|y\|_{L^p_{L^p(\Gamma)}}^p = \int_0^T\int_{\Gamma(t)}|y(x,t)|^p = \int_0^T\int_{\Gamma_0}|\wt y(x,t)|^p|\det D\Phi_t| \leq C_{\Phi}\|\wt y\|_{L^p(\Gamma_0\times(0,T))}^p.
			\end{equation*}
			\item Case $s=1$. As for the first spatial derivatives, we have
			\begin{align*}
			\|\uD_i y\|_{L^p_{L^p(\Gamma)}}^p &= \int_0^T\int_{\Gamma_0}|G^{ij}\uD_j\wt y|^p|\det D\Phi_t|\\
			&\leq C_{\Phi,G}\sum_{j=1}^{n}\int_0^T\int_{\Gamma_0}|\uD_j \wt y|^p \leq C_{\Phi,G,n}\|\na_\Gamma \wt y\|_{L^p(\Gamma_0\times (0,T))}^p.
			\end{align*}
			\item Case $s=2$. Finally, for the second derivatives, we have
			\begin{align*}
			(\uD_j\uD_i y)\circ \Phi_t &= G^{jk}\uD_k(\uD_iy\circ \Phi_t) = G^{jk}\uD_k\left(G^{il}\uD_l\wt y\right)\\ &= G^{jk}\uD_kG^{il}\uD_l\wt{y} + G^{jk}G^{il}\uD_k\uD_l\wt{y},
			\end{align*}
			from which it follows
			\begin{align*}
			\|\uD_j\uD_i y\|_{L^p_{L^p(\Gamma)}}^p &= \int_0^T\int_{\Gamma_0}\left(|G^{jk}\uD_kG^{il}\uD_l\wt y|^p + |G^{jk}G^{il}\uD_k\uD_l\wt y|^p \right)|\det \Phi_t|\\
			&\le C_{\Phi,G,n}\left(\|\na_{\Gamma}\wt y\|_{L^p(\Gamma_0\times(0,T))}^p + \|\na^2_{\Gamma}\wt y\|_{L^p(\Gamma_0\times(0,T))}^p \right).
			\end{align*}
			\item For the time derivative, by noting that $y$ solves \eqref{evol} it follows from the above that
			\begin{equation*}
			\|\dot y\|_{L^p_{L^p(\Gamma)}} \leq \|f\|_{L^p_{L^p(\Gamma)}} + C_{\Phi,G,n,\VV}\|\wt y\|_{W^{2,1}_p(\Gamma_0\times(0,T))}.
			\end{equation*}
		\end{itemize}
		Combining all the different cases above with \eqref{ee9} then gives
		\begin{equation*}
		\|\dot y\|_{L^p_{L^p(\Gamma)}} + \sum_{s=0}^2 \|\pa_{x,\Gamma}^sy\|_{L^p_{L^p(\Gamma)}}\leq C_{\Phi,G,n,\VV}\|\wt y\|_{W^{2,1}_p(\Gamma_0\times(0,T))} \leq C_{T,\Phi,G,n,\VV}\left(\|y\|_{L^p_{L^p(\Gamma)}} + \|f\|_{L^p_{L^p(\Gamma)}} \right),
		\end{equation*}
		as desired.
	\end{proof}	
	\subsection{Duality methods in evolving surfaces}\label{subs:duality}
	In this subsection we establish a duality lemma to establish $L^2_{L^2}$-boundedness of $w$ and $z$.	 Such results are well known in fixed domains, see e.g. \cite{pierre2010global}. Here we prove that it also holds in moving interfaces. We need the following differentiation formulas whose derivations are standard and therefore omitted.  
	
	\newcommand{\tgrad}{\nabla_\Gamma}
	\newcommand{\md}{\partial^\bullet_\Gamma}
	
	\begin{lemma}\label{lem:derivinnerprod}
		For appropriate $u, v$ the following formulas hold:
		\begin{align}
		\dfrac{d}{dt} \int_{\O(t)} u v &= \int_{\O(t)} \dot u \, v + \int_{\O(t)} u \, \dot v + \int_{\O(t)} u \, v \, \nabla \cdot \mathbf V_p \label{eq:transp1} \\
		\dfrac{d}{dt} \intG u v &= \intG \dot u \, v + \intG u \, \dot v + \intG u \, v \, \nabla \cdot \mathbf V_p \label{eq:transp2} \\
		\dfrac{d}{dt} \intG \naG u \cdot \naG v &= \intG \naG \dot u \cdot \naG v + \intG \naG u \cdot \naG \dot v + \intG \mathbf B(\mathbf V_p) \naG u \cdot \naG v \label{eq:transp3}
		\end{align}
		where, for a vector field $v=(v_1, v_2, v_3)$, $\mathbf{B}(v)= (\nabla_\Gamma \cdot v)I - 2D(v)$ and the matrix $D(v)=(D_{ij}(v))_{i,j=1}^{3}$ is defined by
		\begin{align*}
		D_{ij}(v) = \dfrac{1}{2}\sum_{k=1}^{3} \delta_{ik} \underline{D}_k v_j + \delta_{jk} \underline{D}_k v_i  = \dfrac{\underline{D}_iv_j + \underline{D}_jv_i}{2}.
		\end{align*}
	\end{lemma}
	
	\begin{remark}
		The identities in \eqref{eq:transp1}, \eqref{eq:transp2} are well known; the extra term arises from the fact that the domains of integration are themselves time-dependent. Indeed, $\tgrad\cdot\mathbf V_p$ is a curvature term and accounts for local stretching/compression of the domains. As for \eqref{eq:transp3}, the tensor $\mathbf{B}(\mathbf V_p)$ has two components; the first involves $\tgrad\cdot\mathbf V_p$ and arises as in \eqref{eq:transp1}, \eqref{eq:transp2}, and the second term, the deformation tensor $D(\mathbf V_p)$, shows up because the gradient operator itself is time-dependent. We refer for instance to \cite[Appendix A]{elliott2007} for a detailed proof of the result above. 
	\end{remark}
	
	The following duality lemma shows that $w, z$ are bounded in $L^2_{L^2(\Gamma)}$ if they satisfy the inequality \eqref{sum_wz}.
	
	\begin{lemma}\label{duality-L2}
		Assume that $w, z$ are nonnegative and satisfy \eqref{sum_wz}. 
		Then
		\begin{equation*}
		\|w\|_{L^2_{L^2(\GG)}} + \|z\|_{L^2_{L^2(\GG)}} \leq C(T, \|w_0\|_{L^2(\GG_0)}, \|z_0\|_{L^2(\GG_0)}).
		\end{equation*}
	\end{lemma}

	\begin{proof}
		Setting $Z = w + z$ and $A(x,t) = (\dG w + \dGG z)/(w+z)$ we have
		\begin{equation*}
		0< \min\{\dG,\dGG\} \leq A(x,t) \leq \max\{\dG,\dGG\} <+\infty.
		\end{equation*}
		From \eqref{sum_wz} we have
		\begin{equation*}
		\dot{Z} + Z\naG \cdot \VV_p - \Delta_\GG(A(x,t)Z) + \naG\cdot(\J_\GG Z) \leq 0
		\end{equation*}
		or equivalently
		\begin{equation}\label{eq_Z}
		\dot Z - \Delta_\GG (A(x,t)Z) + \naG\cdot(\VV_\GG Z) \leq 0.
		\end{equation}
		Let $\Theta \geq 0$ and $\varphi$ be the nonnegative solution to the dual equation
		\begin{equation}\label{dual}
		\begin{cases}
		-(\dot \varphi + A(x,t)\Delta_\GG \varphi + \naG \varphi\cdot \VV_\GG) = \Theta, &\text{ in } \GG(t),\\
		\varphi(x,T) = 0, &\text{ in } \GG(T).
		\end{cases}
		\end{equation}
		By using integration by parts we have
		\begin{equation}\label{d1}
		\begin{aligned}
		\intT\intG Z\Theta &= -\intT\intG Z(\dot \varphi + A(x,t)\Delta_\GG \varphi + \naG \varphi\cdot \VV_\GG)\\
		&=\int_{\GG_0}Z(0)\varphi(0) + \intT\intG\varphi(\dot Z - \Delta_\GG(A(x,t)Z) + \naG\cdot(\VV_\GG Z))\\
		& \leq \int_{\GG_0}Z(0)\varphi(0) \overset{\text{H\"older ineq.}}{\leq} \|Z(0)\|_{L^2(\GG_0)}\|\varphi(0)\|_{L^2(\GG_0)}.
		\end{aligned}
		\end{equation}
		We now aim to estimate $\|\varphi(0)\|_{L^2(\GG_0)}$. By changing time variable $t \to T - t$ we rewrite the dual equation \eqref{dual} as
		\begin{equation} \label{dual_re}
		\dot \varphi - A(x,t)\Delta_\GG \varphi - \naG\varphi \cdot \VV_\GG = \Theta \text{ in } \GG(t), \qquad \varphi(x,0) = 0 \text{ in } \GG_0.
		\end{equation}
		Multiplying \eqref{dual_re} by $-\Delta_\GG \varphi$ and using Lemma \ref{lem:derivinnerprod} yields 
		\begin{equation}\label{d2}
		\begin{aligned}
		\frac 12 \frac{d}{dt}\intG |\naG \varphi|^2 &+ \intG A|\Delta_\GG \varphi|^2= - \intG \Delta_\GG\varphi \naG \varphi\cdot \VV_\GG  - \intG \Theta \Delta_\GG \varphi {+ \intG \mathbf B(\mathbf V_p) |\naG \varphi|^2}
		\end{aligned}
		\end{equation}
		Since $A(x,t) \geq a:= \min\{\dG,\dGG \} > 0$ we have
		\begin{equation*}
		\intG A|\Delta_\GG\varphi|^2 \geq a\intG|\Delta_\GG \varphi|^2.
		\end{equation*}
		By Young's inequality
		\begin{equation*}
		\left|\intG \Delta_\GG\varphi \naG \varphi\cdot \VV_\GG \right| \leq \|\VV_\GG\|_\infty\intG |\Delta_\GG \varphi||\naG \varphi| \leq \frac{a}{4}\intG|\Delta_\GG \varphi|^2 + C\intG|\naG \varphi|^2,
		\end{equation*}
		similarly
		\begin{equation*}
		\left|\intG \Theta \Delta_\GG \varphi \right| \leq \frac{a}{4}\intG|\Delta_\GG \varphi|^2 + C\intG|\Theta|^2
		\end{equation*}
		{and for the last term we obtain
			\begin{equation*}
			\intG \mathbf B(\mathbf V_p) |\naG \varphi|^2 \leq C \intG |\naG \varphi|^2.
			\end{equation*}
		} Hence, we get from \eqref{d2} that
		\begin{equation}\label{d3}
		\frac{d}{dt}\intG|\naG\varphi|^2 + a\intG|\Delta_\GG\varphi|^2 \leq C\intG|\naG\varphi|^2 + C\intG|\Theta|^2.
		\end{equation}
		Integrating over $(0,t)$ yields  in particular, recalling $\varphi(x,0) = 0$,
		\begin{equation*}
		\|\naG\varphi(t)\|_{L^2(\GG(t))}^2 \leq C\int_0^t\|\naG\varphi(s)\|_{L^2(\GG(s))}^2 + C\|\Theta\|_{L^2_{L^2(\GG)}}^2.
		\end{equation*}
		By Gronwall's inequality
		\begin{equation*}
		\|\naG\varphi\|_{L^\infty_{L^2(\GG)}}^2 \leq C(T)\|\Theta\|_{L^2_{L^2(\GG)}}^2.
		\end{equation*}
		Coming back to \eqref{d3} we obtain
		\begin{equation*}
		\|\Delta_\GG \varphi\|_{L^2_{L^2(\GG)}} \leq C(T)\|\Theta\|_{L^2_{L^2(\GG)}}.
		\end{equation*}
		Therefore from \eqref{dual_re} it follows
		\begin{equation*}
		\|\dot \varphi\|_{L^2_{L^2(\GG)}} \leq C(T)\|\Theta\|_{L^2_{L^2(\GG)}}
		\end{equation*}
		and obviously
		\begin{equation*}
		\|\varphi\|_{L^2_{L^2(\GG)}}\leq C(T)\|\Theta\|_{L^2_{L^2(\GG)}}.
		\end{equation*}
		We can now estimate
		\begin{equation*}
		\begin{aligned}
		\|\varphi(0)\|_{L^2(\GG_0)}^2 &= -\int_0^T\pa_t\|\varphi(t)\|_{L^2(\GG(t))}^2\\
		&= -2\int_0^T\int_{\GG(t)}\varphi \dot \varphi  {- \int_0^T\intG \varphi^2 \naG\cdot \mathbf V_p }\\
		&\leq 2 \|\varphi\|_{L^2_{L^2(\GG)}}\|\dot\varphi\|_{L^2_{L^2(\GG)}} { + C \|\varphi\|^2_{L^2_{L^2(\GG)}}}\\
		&\leq C(T)\|\Theta\|_{L^2_{L^2(\GG)}}^2.
		\end{aligned}
		\end{equation*}
		Inserting this into \eqref{d1} leads  to
		\begin{equation*}
		\int_0^T\intG Z\Theta \leq C(T)\|Z(0)\|_{L^2(\GG_0)}\|\Theta\|_{L^2_{L^2(\GG)}}
		\end{equation*}
		for arbitrary $0\leq \Theta \in L^2_{L^2(\GG)}$ and therefore by duality
		\begin{equation*}
		\|Z\|_{L^2_{L^2(\GG)}} \leq C(T)\|w_0 + z_0\|_{L^2(\GG_0)}.
		\end{equation*}
		Recalling $Z = w+z$ and $w,z$ are nonnegative we can finish the proof of Lemma \ref{duality-L2}.
	\end{proof}
	
	The next duality lemma shows that any regularity of $w$ and be transferred to $z$ through the inequality \eqref{sum_wz}. The $L^p$-maximal regularity result in the previous subsection plays an important role in the following proof.
	\begin{lemma}\label{duality-1}
		Assume that $w, z$ are nonnegative and satisfy \eqref{sum_wz}. Then, for any $1<p<\infty$, there exists a constant {$C(T)>0$} such that
		\begin{equation*}
		\|z\|_{L^p_{L^p(\GG)}} \leq {C(T)}(\|z_0+w_0\|_{L^p(\GG_0)} + \|w\|_{L^p_{L^p(\GG)}}).
		\end{equation*}
	\end{lemma}
	\begin{proof}
		Fix $0\leq  \Phi \in L^{p'}_{L^{p'}(\GG)}$ where $p' = p/(p-1)$, and let $\eta$ be the nonnegative solution to 
		\begin{equation*}
		\begin{aligned}
		-(\dot \eta + \dGG \Delta_\GG \eta + \naG \eta\cdot \VV_\GG) &= \Phi, \text{ on } \GGG_T,\\
		\eta(T) &= 0.
		\end{aligned}
		\end{equation*}
		From \eqref{sum_wz} it follows that
		\begin{equation*}
		\dot z - \dGG\Delta_\GG z + \naG\cdot(\VV_\GG z) \leq -\dot w + \dG\Delta_\GG w - \naG\cdot(\VV_\GG w).
		\end{equation*}
		Therefore, using integration by parts we have
		\begin{equation}\label{l4:e1}
		\begin{aligned}
		\int_{\GGG_T}z\Phi   &= -\int_{\GGG_T}z(\dot \eta + \dGG \Delta_\GG \eta + \naG \eta\cdot \VV_\GG) \\
		&= \int_{\GG_0}z_0\eta(0)  +\int_{\GGG_T}\eta(\dot z - \dGG\Delta_\GG z + \naG\cdot(z\VV_\GG)) \\
		&\leq \int_{\GG_0}z_0\eta(0)  -\int_{\GGG_T}\eta(\dot w - \dG\Delta_\GG w + \naG\cdot(w\VV_\GG)) \\
		&= \int_{\GG_0}(z_0+w_0)\eta(0)  +\int_{\GGG_T}w(\dot\eta + \dG\Delta_\GG \eta + \naG\eta \cdot \VV_\GG) \\
		&= \int_{\GG_0}(z_0+w_0)\eta(0)  -\int_{\GGG_T}w\Phi  \\
		&\leq \|z_0+w_0\|_{L^p(\GG_0)}\|\eta(0)\|_{L^{p'}(\GG_0)} + \|w\|_{L^p_{L^p(\GG)}}\|\Phi\|_{L^{p'}_{L^{p'}(\GG)}}.
		\end{aligned}
		\end{equation}	
		By the maximum regularity in Lemma \ref{max-regular} (or more precisely Remark \ref{backward}) we have
		\begin{equation*}
		\|\eta\|_{L^{p'}_{L^{p'}(\GG)}} + \|\pa_t\eta\|_{L^{p'}_{L^{p'}(\GG)}} + \|\Delta_\GG\eta\|_{L^{p'}_{L^{p'}(\GG)}} \leq {C(T)}\|\Phi\|_{L^{p'}_{L^{p'}(\GG)}}.
		\end{equation*}
		Therefore, by using H\"older inequality,
		\begin{equation*}
		\begin{aligned}
		\|\eta(0)\|_{L^{p'}(\GG_0)}^{p'} &= -\int_0^T\pa_t\int_{\GG(t)}|\eta(t)|^{p'} \\
		&= -p'\int_0^T\intG \eta^{p'-1}\pa_t \eta  \\
		&\leq p'\|\pa_t\eta\|_{L^{p'}_{L^{p'}(\GG)}}\|\eta\|_{L^{p'}_{L^{p'}(\GG)}}^{p'-1}\\
		&\leq {C(T)}\|\Phi\|_{L^{p'}_{L^{p'}(\GG)}}^{p'}.
		\end{aligned}
		\end{equation*}
		Hence, we get from \eqref{l4:e1}
		\begin{equation*}
		\int_{\GGG_T}z\Phi   \leq {C(T)}(\|z_0+w_0\|_{L^p(\GG_0)} + \|w\|_{L^p_{L^{p}(\GG)}})\|\Phi\|_{L^{p'}_{L^{p'}(\GG)}}
		\end{equation*}
		and by duality we obtain the desired estimate.
	\end{proof}
	
	\subsection{Heat regularisation}\label{subs:heat_regularisation}
	This subsection concerns the regularisation effect of the heat operator either in a moving surfaces without boundaries or domains with Neumann boundary conditions. We start with a nonlinear Gronwall inequality.
	\begin{lemma}\label{non-Gron}
		Let $y: \R_+ \to \R_+$ satisfy
		\begin{equation*}
		y' \leq C_1 + \alpha(t)y^{1-r} \quad \text{ for all } \quad t\in (0,T),\\
		\end{equation*}
		for some $r\in (0,1)$ and $\alpha$ is a nonnegative function. Then
		\begin{equation*}
		\sup_{t\in (0,T)}y(t) \leq C_2\left[y(0) + (C_1+1)T + \left(\int_0^T\alpha(t)\right)^{1/r}\right].
		\end{equation*}
	\end{lemma}
	\begin{proof}
		By Young's inequality, for any $\delta >0$,
		\begin{equation*}
		y^{1-r} \leq \delta y + r\left(\frac{1-r}{\delta}\right)^{\frac{1-r}{r}},
		\end{equation*}
		it follows that
		\begin{equation*}
		y' \leq \delta \alpha(t)y + C_1 + r\alpha(t)\left(\frac{1-r}{\delta}\right)^{\frac{1-r}{r}}=: \delta \alpha(t)y + \beta(t).
		\end{equation*}
		By Gronwall's inequality we have
		\begin{equation*}
		y(t) \leq y(0)\exp\left(\delta\int_0^t\alpha(\tau)  \right) + \int_0^t\beta(s)\exp\left(\int_s^t\delta\alpha(\tau)  \right) .
		\end{equation*}
		We choose $\delta = \left(\int_0^t\alpha(\tau) \right)^{-1}$ and thus $\delta \int_s^t\alpha(\tau) \leq 1$ for all $s\in (0,t)$. Therefore,
		\begin{equation*}
		\begin{aligned}
		y(t) &\leq ey(0) + e\int_0^t\beta(s) \\
		&= ey(0) + e\int_0^t\left[C_1 + r(1-r)^{\frac{1-r}{r}}\delta^{\frac{r-1}{r}}\int_0^t\alpha(s) \right]\\
		&\leq C_2\left[y(0) + (C_1+1)T + \left(\int_0^T\alpha(\tau) \right)^{1/r}\right]
		\end{aligned}
		\end{equation*}
		for all $t\in (0,T)$ with some $C_2>0$ independent of $T$, which finishes the proof.
	\end{proof}	
	
	\begin{lemma}\label{heat-regularity}
		Let $f\in L^p_{L^p(\GG)}$ satisfy $\|f\|_{L^p_{L^p(\GG)}} \leq C(T)$ for some $p\in (1,\infty)$. Let $z$ be the solution to 
		\begin{equation*}
		\begin{aligned}
		\dot{z} + z\naG\cdot \VV_p - \delta \Delta_\GG z + \naG\cdot(\J_\GG z) &= f, \quad \text{ on } \quad \GGG_T,\\
		z(0)&= z_0.
		\end{aligned}
		\end{equation*}
		Assume moreover that $z_0 \in L^\infty(\GG_0)$ and $\|z\|_{L^\infty_{L^1(\GG)}} \leq M$ for some constant $M>0$. Then
		\begin{equation*}
		\|z\|_{L^r_{L^r(\GG)}} \leq C(T) \quad \text{ for all } r\in [1,s)
		\end{equation*}
		where 
		\begin{equation*}
		s = \begin{cases}
		<+\infty\; \text{ arbitrary } \quad &\text{ if } \quad p \geq (d+2)/2,\\
		\frac{p(d+2)}{d+2-2p} \quad &\text{ if } \quad p < (d+2)/2.
		\end{cases}
		\end{equation*}
	\end{lemma}
	\begin{proof}
		
		In this proof, for the sake of brevity, we will denote by $\|\cdot\|_{L^p}$ and $\|\cdot\|_{H^1}$ the usual norms of $L^p(\GG(t))$ and $H^1(\GG(t))$ respectively.
		
		\medskip
		\noindent\underline{Step 1. Energy estimates.}	
		Let $\mu>1$. 
		Multiplying the equation of $z$ by $\mu |z|^{\mu - 2}z$ (more precisely, a smooth version of this function and let the regularisation go to zero) we obtain 
		\begin{equation*}
		\mu\intG \left(\dot{z} + z\naG\cdot \VV_p - \delta \Delta_\GG z + \naG\cdot(\J_\GG z)\right) |z|^{\mu-2} z = \mu \intG f z |z|^{\mu -2},
		\end{equation*}		
		Using Lemma \ref{lem:derivinnerprod} we have, for the time derivative term, the estimate
		\begin{equation}
		\mu \intG \dot z \, z \, |z|^{\mu-2} = \dfrac{d}{dt} \intG |z|^{\mu} - \intG |z|^\mu \naG \cdot \mathbf V_p
		\end{equation}
		and for the term involving $\naG \cdot \mathbf V_p$ we get
		\begin{align*}
		\mu \intG z \naG \cdot \mathbf V_p \, z \, |z|^{\mu -2} = \mu \intG |z|^\mu \, \naG \cdot \mathbf V_p.
		\end{align*}
		Combining these two terms gives 
		\begin{align*}
		\mu \intG \dot z \, z \, |z|^{\mu-2} + \mu \intG z \naG \cdot \mathbf V_p \, z \, |z|^{\mu -2} &= \partial_t \intG |z|^{\mu} + (\mu - 1) \intG |z|^\mu \naG \cdot \mathbf V_p \\
		&=\partial_t \intG |z|^{\mu} - \mu (\mu - 1) \intG z |z|^{\mu-2} \naG z \cdot \mathbf V_p,
		\end{align*}	
		For the term with $\mathbf J_\Gamma$ we obtain, using the divergence theorem 
		\begin{align*}
		\mu \intG \naG \cdot (\mathbf J_\Gamma z) z |z|^{\mu -2} &= -\mu(\mu-1) \intG \mathbf (\mathbf J_\Gamma \cdot \naG z) \, z |z|^{\mu -2} \\
		&=-\mu(\mu-1) \intG \mathbf (\mathbf V_\Gamma \cdot \naG z) \, z |z|^{\mu -2} + \mu(\mu-1) \intG \mathbf (\mathbf V_p \cdot \naG z) \, z |z|^{\mu -2}
		\end{align*}
		Finally, for the second order term, we have 
		\begin{align*}
		-\delta \mu \intG (\Delta_\Gamma z) z |z|^{\mu-2} &= \delta \mu \intG \naG z \cdot \naG (z|z|^{\mu-2}) \\
		&= \delta \mu (\mu-1) \intG |z|^{\mu -2} \, |\naG z|^2 \\
		&= \dfrac{4(\mu-1) \delta}{\mu} \intG | \naG \left( |z|^{\mu/2} \right) |^2.
		\end{align*}
		All in all, we have
		\begin{equation}\label{l4:e2}
		\begin{aligned}
		\pa_t\intG |z|^{\mu}  + \frac{4(\mu-1)\delta}{\mu} &\intG |\naG (z^{\mu/2})|^2 \\
		&= \mu(\mu-1) \intG (\mathbf V_\Gamma \cdot \naG z) \, z |z|^{\mu -2} + \mu\intG f|z|^{\mu-2}z  
		\end{aligned}
		\end{equation}
		We estimate the first term on the right hand side of \eqref{l4:e2} as
		\begin{equation}\label{l4:e3}
		\begin{aligned}
		&\mu(\mu-1)\intG (\mathbf V_\Gamma \cdot \naG z) \, z |z|^{\mu -2}  \\
		&\leq \mu(\mu-1){\|\VV_\GG\|_{\infty}}\intG|\naG z||z|^{\mu-1} \\
		&= 2(\mu-1){\|\VV_\GG\|_\infty}\intG|\naG (z^{\mu/2})||z|^{\mu/2} \\
		&\leq \frac{2\delta(\mu-1)}{\mu}\intG|\naG (z^{\mu/2})|^2  + \frac{\mu(\mu-1){\|\VV_\GG\|_\infty}}{2\delta}\intG |z|^{\mu} 
		\end{aligned}
		\end{equation}
		where we used Cauchy-Schwarz's inequality at the last step. For the second term on the right hand side of \eqref{l4:e2} we use H\"older's inequality to get
		\begin{equation}\label{l4:e4}
		\mu\intG f|z|^{\mu-2}z  \leq \mu\|f\|_{L^{p}}\|z\|_{L^{\frac{p(\mu-1)}{p-1}}}^{\mu-1}.
		\end{equation}
		Adding $\frac{2\delta(\mu-1)}{\mu}\intG|z|^\mu  $ to both sides of \eqref{l4:e2} then using \eqref{l4:e3} and \eqref{l4:e4} we obtain
		\begin{equation}\label{l4:e5}
		\pa_t\|z\|_{L^\mu}^\mu +C\|z^{\mu/2}\|_{H^1}^2 \leq {C(T)}\|z\|_{L^\mu}^\mu + C\|f\|_{L^p}\|z\|_{L^{\frac{p(\mu-1)}{p-1}}}^{\mu-1}.
		\end{equation}
		From {Sobolev's embedding} we know that
		\begin{equation*}
		\|z^{\mu/2}\|_{H^1}^2 \geq C_S\|z^{\mu/2}\|_{L^{2^*}}^2 = C_S\|z\|_{L^{\mu 2^*/2}}^{\mu}
		\end{equation*}
		where
		\begin{equation}\label{def-s}
		2^* = \begin{cases}
		<+\infty \quad \text{ arbitrary } \quad \text{ if } \quad d\leq 2,\\
		\frac{2d}{d-2} \quad \text{ if } \quad d\geq 3.
		\end{cases}
		\end{equation}
		Note that we always have $\mu 2^* > \mu$. It then follows from \eqref{l4:e5} that 
		\begin{equation}\label{l4:e6}
		\pa_t\|z\|_{L^\mu}^\mu +C\|z\|_{L^{\mu 2^*/2}}^\mu \leq {C(T)}\|z\|_{L^\mu}^\mu + C\|f\|_{L^p}\|z\|_{L^{\frac{p(\mu-1)}{p-1}}}^{\mu-1}.
		\end{equation}
		Using {interpolation}, the assumption $\|z\|_{L^\infty_{L^1(\GG)}}\leq M$, and Young's inequality we have
		\begin{equation*}
		{C(T)}\|z\|_{L^\mu}^\mu \leq {C(T)}\|z\|_{L^1}^{\theta\mu}\|z\|_{L^{\mu 2^*/2}}^{(1-\theta) \mu} \leq M^{\theta \mu}{C(T)}\|z\|_{L^{\mu 2^*/2}}^{\theta \mu} \leq C_{\varepsilon,M}{C(T)} + \varepsilon\|z\|_{L^{\mu 2^*/2}}^{\mu}
		\end{equation*}
		with
		\begin{equation*}
		\frac{1}{\mu} = \frac{\theta}{1} + \frac{1-\theta}{\mu 2^*/2} \quad \text{ which implies } \quad \theta = \frac{2^*-2}{\mu 2^*-2} \in (0,1).
		\end{equation*}
		By choosing $\varepsilon$ small enough we get from \eqref{l4:e6} the energy estimate
		\begin{equation}\label{l4:e7}
		\pa_t\|z\|_{L^\mu}^\mu +C\|z\|_{L^{\mu 2^*/2}}^\mu \leq {C(T)} + C\|f\|_{L^p}\|z\|_{L^{\frac{p(\mu-1)}{p-1}}}^{\mu-1}.
		\end{equation}
		
		\medskip
		\noindent\underline{Step 2. Initial $\mu = p$.}	
		Denote $p_0:= p$ and let $\mu = p_0$ in \eqref{l4:e7}, we have
		\begin{equation}\label{l4:e8}
		\pa_t\|z\|_{L^{p_0}}^{p_0} + C\|z\|_{L^{s_0}}^{p_0} \leq {C(T)} + C\|f\|_{L^{p_0}}\|z\|_{L^{p_0}}^{p_0-1}
		\end{equation}
		where $s_0 = p_02^*/2$ (see \eqref{def-s}). Applying Lemma \ref{non-Gron} to \eqref{l4:e8} with $y = \|z\|_{L^{p_0}}^{p_0}$, $r = 1/p_0 \in (0,1)$ and $\alpha(t) = \|f(t)\|_{L^{p_0}}$ we have
		\begin{equation}\label{C(T)0}
		\begin{aligned}
		\sup_{t\in(0,T)}\|z(t)\|_{L^{p_0}}^{p_0} &\leq C\left[\|z_0\|_{L^{p_0}}^{p_0} + (C(T)+1)T + \left(\int_0^T\|f(\tau)\|_{L^{p_0}} \right)^{p_0}\right]\\
		&\leq C\left[\|z_0\|_{L^{p_0}}^{p_0} + C(T) + \|f\|_{L^{p_0}_{L^{p_0}(\GG)}}^{p_0}T^{p_0-1}\right] =: C_{T,0}
		\end{aligned}
		\end{equation}
		where we used H\"older's inequality at the last step. Now integrating \eqref{l4:e8} on $(0,T)$ yields 
		\begin{equation}\label{D_T0}
		\begin{aligned}
		\int_0^T\|z\|_{L^{s_0}}^{p_0}  &\leq C\left[\|z_0\|_{L^{p_0}(\GG_0)}^2 + C(T) + \int_0^T\|f\|_{L^{p_0}}\|z\|_{L^{p_0}}^{p_0-1} \right]\\
		&\leq C\left[\|z_0\|_{L^{p_0}(\GG_0)}^2 + C(T) + \int_0^T\|f\|_{L^{p_0}}^{p_0}  + \int_0^T\|z\|_{L^{p_0}}^{p_0} \right]\\
		&\leq C\left[\|z_0\|_{L^{p_0}(\GG_0)}^2 + C(T) + \|f\|_{L^{p_0}_{L^{p_0}(\GG)}}^{p_0} + TC_{T,0}\right] =: D_{T,0}.
		\end{aligned}
		\end{equation}
		
		\medskip
		\noindent\underline{Step 3. Iteration.}	
		Assume now that we have the exponents $\{p_n\}$ and $\{s_n = p_n2^*/2\}$ and the constants $C_{T,n}$ and $D_{T,n}$ depending continuously on $T$, such that
		\begin{equation*}
		\sup_{t\in(0,T)}\|z(t)\|_{L^{p_n}}^{p_n} \leq C_{T,n} \quad \text{ and } \quad \int_0^T\|z\|_{L^{s_n}}^{p_n}  \leq D_{T,n}.
		\end{equation*}
		Let $p_{n+1}$ be determined later. We choose $\mu = p_{n+1}$ in \eqref{l4:e7} and denote $s_{n+1} = p_{n+1}s(p_{n+1})$ (see \eqref{def-s}) to get
		\begin{equation}\label{l4:e9}
		\pa_t\|z\|_{L^{p_{n+1}}}^{p_{n+1}} + C\|z\|_{L^{s_{n+1}}}^{p_{n+1}} \leq {C(T)} + C\|f\|_{L^{p_0}}\|z\|_{L^{\frac{p_0(p_{n+1}-1)}{p_0-1}}}^{p_{n+1}-1}.
		\end{equation}
		Choose $p_{n+1}$ such that
		\begin{equation*}
		p_{n+1} < \frac{p_0(p_{n+1}-1)}{p_0-1} < s_n,
		\end{equation*}
		and thus, we can use the {interpolation} inequality
		\begin{equation*}
		\|z\|_{L^{\frac{p_0(p_{n+1}-1)}{p_0-1}}} \leq \|z\|_{L^{p_{n+1}}}^{1-\theta}\|z\|_{L^{s_n}}^{\theta} 
		\end{equation*}
		where
		\begin{equation}\label{theta1}
		\frac{p_0-1}{p_0(p_{n+1}-1)} = \frac{1-\theta}{p_{n+1}} + \frac{\theta}{s_n}.
		\end{equation}
		It follows then from \eqref{l4:e9}
		\begin{equation}\label{l4:e9_1}
		\pa_t\|z\|_{L^{p_{n+1}}}^{p_{n+1}} + C\|z\|_{L^{s_{n+1}}}^{p_{n+1}} \leq {C(T)} + C\|f\|_{L^{p_0}}\|z\|_{L^{s_n}}^{\theta(p_{n+1}-1)}\|z\|_{L^{p_{n+1}}}^{p_{n+1} - (1+\theta(p_{n+1}-1))}.
		\end{equation}
		Applying Lemma \ref{non-Gron} with $y = \|z\|_{L^{p_{n+1}}}^{p_{n+1}}$ and $r = 1 - (1+\theta(p_{n+1}-1))/p_{n+1} \in (0,1)$ we have
		\begin{equation}\label{l4:e10}
		\sup_{t\in(0,T)}\|z(t)\|_{L^{p_{n+1}}}^{p_{n+1}} \leq C\left[\|z_0\|_{L^{p_{n+1}}}^{p_{n+1}} + {(C(T)+1)}T + \left(\int_0^T\|f\|_{L^{p_0}}\|z\|_{L^{s_n}}^{\theta(p_{n+1}-1)}  \right)^{\frac{p_{n+1}}{1+\theta(p_{n+1}-1)}} \right].
		\end{equation}
		By H\"older's inequality
		\begin{equation*}
		\int_0^T\|f\|_{L^{p_0}}\|z\|_{L^{s_n}}^{\theta(p_{n+1}-1)}  \leq \|f\|_{L^{p_0}_{L^{p_0}(\GG)}}\left(\int_0^T\|z\|_{L^{s_n}}^{\theta(p_{n+1}-1)\frac{p_0}{p_0-1}}  \right)^{\frac{p_0-1}{p_0}}.
		\end{equation*}
		Now we choose $p_{n+1}$ such that
		\begin{equation}\label{theta2}
		\theta(p_{n+1}-1)\frac{p_0}{p_0-1} = p_n,
		\end{equation}
		which in combination with \eqref{theta1} implies
		\begin{equation*}
		\theta = 1 - \frac{2}{d}\frac{p_0-1}{p_0}\frac{p_{n+1}}{p_{n+1}-1} \in (0,1).
		\end{equation*}
		Therefore 
		\begin{equation}\label{l4:e11}
		\int_0^T\|f\|_{L^{p_0}}\|z\|_{L^{s_n}}^{\theta(p_{n+1}-1)}  \leq \|f\|_{L^{p_0}_{L^{p_0}(\GG)}}\left(\int_0^T\|z\|_{L^{s_n}}^{p_n}  \right)^{\frac{p_0-1}{p_0}} \leq \|f\|_{L^{p_0}_{L^{p_0}(\GG)}}D_{T,n}^{\frac{p_0 -1}{p_0}}.
		\end{equation}
		Thus it follows from \eqref{l4:e10} that
		\begin{equation}\label{C(T)n+1}
		\sup_{t\in(0,T)}\|z(t)\|_{L^{p_{n+1}}}^{p_{n+1}} \leq C\left[\|z_0\|_{L^{p_{n+1}}}^{p_{n+1}} + {C(T)} + \left(\|f\|_{L^{p_0}_{L^{p_0}(\GG)}}D_{T,n}^{\frac{p_0 -1}{p_0}} \right)^{\frac{p_{n+1}}{1+\theta(p_{n+1}-1)}} \right] =: C_{T,n+1}.
		\end{equation}
		Now integrating \eqref{l4:e9_1} on $(0,T)$ and using \eqref{l4:e11} and \eqref{C(T)n+1} we get
		\begin{equation}\label{D_Tn+1}
		\begin{aligned}
		\int_0^T\|z\|_{L^{s_{n+1}}}^{p_{n+1}}  &\leq C\left[\|z_0\|_{L^{p_{n+1}}}^{p_{n+1}} + {C(T)} + C_{T,n+1}^{1 - \frac{1+\theta(p_{n+1}-1)}{p_{n+1}}}\int_0^T\|f\|_{L^{p_0}}\|z\|_{L^{s_n}}^{\theta(p_{n+1} -1)} \right]\\
		&\leq C\left[\|z_0\|_{L^{p_{n+1}}}^{p_{n+1}} +{C(T)} + C_{T,n+1}^{p_{n+1} - (1+\theta(p_{n+1}-1))}\|f\|_{L^{p_0}_{L^{p_0}(\GG)}}D_{T,n}^{\frac{p_0 -1}{p_0}}\right] =: D_{T,n+1}.
		\end{aligned}
		\end{equation}
		From \eqref{C(T)n+1} and \eqref{D_Tn+1} we get the desired estimates with
		\begin{equation}\label{pn}
		p_{n+1} = p_n\frac{d(p_0-1)}{p_0(d-2)+2} + \frac{dp_0}{p_0(d-2)+2}
		\end{equation}
		due to \eqref{theta1} and \eqref{theta2}.
		
		\medskip
		\noindent\underline{Step 4. Passing to limit.}
		By recalling $s_n < +\infty$ arbitrary when $d\leq 2$ and $s_n = p_n2^*/2$ (see \eqref{def-s}) when $d\geq 3$, an {interpolation} yields 
		\begin{equation*}
		L^\infty_{L^{p_n}(\GG)}\cap L^{p_n}_{L^{s_n}(\GG)} \hookrightarrow L^{p_n(d+2)/2}_{L^{p_n(d+2)/d}(\GG)},
		\end{equation*}
		thus
		\begin{equation}\label{l4:e12}
		\|z\|_{L^{p_n(d+2)/d}_{L^{p_n(d+2)/d}(\GG)}} \leq C(T) \quad \text{ for all } \quad n\geq 0.
		\end{equation}
		It is easy to see from \eqref{pn} that
		\begin{equation*}
		\lim_{n\to \infty}p_n = \begin{cases}
		+\infty &\text{ if } p_0 \geq (d+2)/2,\\
		p_\infty = \frac{dp_0}{d+2-2p_0} &\text{ if } p_0 < (d+2)/2.
		\end{cases}
		\end{equation*}
		Passing to limit $n\to \infty$ in \eqref{l4:e12} we obtain the desired estimate of Lemma \ref{heat-regularity}.
	\end{proof}	
	
	\begin{lemma}\label{L-infinity-Neumann}
		Let $u$ be the nonnegative solution to
		\begin{equation*}
		\begin{aligned}
		\dot{u} + u\na \cdot \VV_p - \dO \Delta u + \na\cdot(\J_\O u) &= 0, &&\text{ in } Q_T,\\
		\dO \na u \cdot \nu - uj &= z, &&\text{ on } \GGG_T,\\
		u(x,0) &= u_0(x), && \text{ in } \O_0.
		\end{aligned}
		\end{equation*}
		Assume that $\|u\|_{L^\infty_{L^1(\O)}}\leq M(T)$, $\|z\|_{L^r_{L^r(\GG)}} \leq C(T)$ for some $r\ge d+2$ and $u_0\in L^\infty(\O_0)$, then $u$ is bounded in $L^\infty_{L^\infty(\O)}$, i.e.
		\begin{equation}
		\|u\|_{L^\infty_{L^\infty(\O)}} \leq C(T).
		\end{equation}
	\end{lemma}
	\begin{proof}
		For each $T>0$ we define
		\begin{equation*}
		V = V(0,T):= \{u:Q_T \to \R: \; \|u\|_{V}^2:= \sup_{t\in (0,T)}\|u(t)\|_{L^2(\O(t))}^2 + \int_0^T\|u\|_{H^1(\O(t))}^2  < +\infty\}.
		\end{equation*}
		We will also using the following anisotropic Sobolev embeddings (see e.g. \cite[II.\S3]{LSU68})
		\begin{equation}\label{inter}
		\|u\|_{L^{2+4/d}_{L^{2+4/d}(\O)}} + \|u\|_{L^{2+2/d}_{L^{2+2/d}(\GG)}} \leq C(T)\|u\|_V.
		\end{equation}
		The dependence on $T$ of the constant here indicates that the embedding constants depend on time through moving domains.
		
		\medskip
		We will first prove that
		\begin{equation*}
		\|u\|_V \leq C(T).
		\end{equation*}
		In order to do that, we multiply the equation of $u$ by $u$ and then integrate over $\O(t)$ to obtain
		\begin{align*}
		\int_{\O(t)} \dot u u + u^2 \nabla \cdot \mathbf V_p - \delta_\Omega u \Delta u + \nabla\cdot (\mathbf J_\Omega u) u = 0.
		\end{align*}
		The transport formula in Lemma \ref{lem:derivinnerprod} leads  to
		\begin{align*}
		\int_{\O(t)} \dot u u = \dfrac{1}{2} \dfrac{d}{dt} \int_{\O(t)} |u|^2 - \dfrac{1}{2}\int_{\O(t)} u^2 \nabla \cdot \mathbf V_p 
		\end{align*}
		and by integration by parts we have
		\begin{align*}
		-\delta_\Omega \int_{\O(t)} (\Delta u) u = \delta_\Omega \int_{\O(t)} |\nabla u|^2 - \delta_\Omega \intG u \nabla u \cdot \nu.
		\end{align*}
		On the other hand, for the terms involving the velocities, we obtain
		\begin{align*}
		\int_{\O(t)} u^2 \nabla \cdot \mathbf V_p = \dfrac{1}{2}\int_{\O(t)} u^2 \nabla \cdot \mathbf V_p - \dfrac{1}{2} \int_{\O(t)} \nabla (u^2) \cdot \mathbf V_p + \dfrac{1}{2}\intG u^2 \mathbf V_p \cdot \nu.
		\end{align*}
		and
		\begin{align*}
		\int_{\O(t)} \nabla \cdot (\mathbf J_\Omega u) u &= -\int_{\O(t)} u \mathbf V_\Omega \cdot \nabla u + \int_{\O(t)} u \mathbf V_p \cdot \nabla u + \intG u^2\, \mathbf J_\Omega \cdot \nu \\
		&=\int_{\O(t)} u \mathbf V_\Omega \cdot \nabla u +\dfrac{1}{2} \int_{\O(t)} \nabla(u^2) \cdot \mathbf V_p + \intG u^2 \, j.
		\end{align*}
		Adding the four identities above and recalling the boundary condition $\delta_\Omega \nabla u\cdot  \nu - uj = z$ leads  to
		\begin{align*}
		\dfrac{1}{2} \dfrac{d}{dt} \int_{\O(t)} |u|^2 + \delta_\Omega \int_{\O(t)} |\nabla u|^2 &= - \dfrac{1}{2} \intG u^2 \mathbf V_p \cdot \nu - \dfrac{1}{2} \int_{\O(t)} u \mathbf V_\Omega \cdot \nabla u + \intG u z.
		\end{align*}
		Noting that $\mathbf V_p|_\Gamma \cdot \nu = \mathbf V_\Gamma \cdot \nu$ we rewrite the above as 
		\begin{align*}
		\dfrac{1}{2} \dfrac{d}{dt} \int_{\O(t)} |u|^2 + \delta_\Omega \int_{\O(t)} |\nabla u|^2 &= - \dfrac{1}{2} \intG u^2 \mathbf V_\Gamma \cdot \nu - \dfrac{1}{2} \int_{\O(t)} u \mathbf V_\Omega \cdot \nabla u + \intG u z.
		\end{align*}
		Using Cauchy's inequality yields 
		\begin{equation*}
		\left| \dfrac{1}{2}\int_{\O(t)}u\VV_\O\cdot\na u  \right| \leq {\dfrac{\|\VV_\O\|_\infty}{2}}\int_{\O(t)}|u||\na u|  \leq \frac{\dO}{4}\int_{\O(t)}|\na u|^2  + {C(T)}\int_{\O(t)}|u|^2 .
		\end{equation*}
		On the other hand, by modified Trace theorem, see \cite[Theorem 1.5.1.10]{grisvard2011elliptic}, we can estimate
		\begin{equation*}
		\left|\dfrac{1}{2} \intG u^2 \mathbf V_\Gamma \cdot \nu\right|	\leq \dfrac{\|\mathbf V_\Gamma\|_{\infty}}{2} \int_{\GG(t)}|u|^2  \leq \frac{\dO}{8}\int_{\O(t)}|\na u|^2  + {C(T)}\int_{\O(t)}|u|^2 ,
		\end{equation*}
		and similarly
		\begin{equation*}
		\int_{\GG(t)}uz  \leq \int_{\GG(t)}|u|^2  + \frac 14\int_{\GG(t)}|z|^2  \leq \frac{\dO}{8}\int_{\O(t)}|\na u|^2  + {C(T)}\int_{\O(t)}|u|^2 + \frac 14\int_{\GG(t)}|z|^2 .
		\end{equation*}
		Therefore, we obtain
		\begin{equation*}
		\frac{d}{dt}\int_{\O(t)}|u|^2  + \frac{\dO}{2}\int_{\O(t)}|\na u|^2  \leq {C(T)}\int_{\O(t)}|u|^2  + C\int_{\GG(t)}|z|^2 .
		\end{equation*}
		Adding both sides with $\frac{\dO}{2}\int_{\O(t)}|u|^2 $ then integrating over $(0,T)$ we get
		\begin{equation}\label{tt}
		\sup_{t\in (0,T)}\|u(t)\|_{L^2(\O(t))}^2 + \frac{\dO}{2}\int_0^T\|u\|_{H^1(\O(t))}^2  \leq \|u_0\|_{L^2(\O_0)}^2 +  {C(T)}\int_0^T\|u\|_{L^2(\O(t))}^2  + C\|z\|_{L^2_{L^2(\GG)}}^2.
		\end{equation}
		Using $\|u\|_{L^\infty_{L^1(\Omega)}}\leq M$ and {interpolation} inequality we have, for some $\theta \in (0,1)$,
		\begin{equation*}
		\begin{aligned}
		{C(T)}\int_0^T\|u\|_{L^2(\O(t))}^2  &\leq {C(T)}\int_0^T\|u\|_{L^1(\O(t))}^{2\theta}\|u\|_{H^1(\O(t))}^{2(1-\theta)} \\
		&\leq {C(T)}M(T)^{2\theta}\int_0^T\|u\|_{H^1(\O(t))}^{2(1-\theta)} \\
		&\overset{\text{Young's inequality}}{\leq}	\frac{\dO}{4}\int_0^T\|u\|_{H^1(\O(t))}^2  +{C(T)}.
		\end{aligned}
		\end{equation*}
		Using this in \eqref{tt} and recalling that $\|z\|_{L^2_{L^2(\GG)}} \leq C(T)$ we get the desired estimate
		\begin{equation*}
		\|u\|_V \leq C(T).
		\end{equation*}
		
		Now to prove the boundedness of $u$, we fix a constant $k>0$ which will be determined later, and define, for each $i \geq 1$,
		\begin{equation*}
		v_i:= (u - k + k/2^i)_+ = \max\{u - k+k/2^i;0\},
		\end{equation*}
		and
		\begin{equation*}
		A_i = \{(x,t)\in Q_T: u(x,t) \geq k - k/2^i\} \quad \text{ and } \quad B_i = \{(x,t)\in \GGG_T: u(x,t) \geq k - k/2^i\}.
		\end{equation*}
		We will frequently use the following observation
		\begin{equation}\label{layer}
		\forall (x,t)\in A_{i+1}\cup B_{i+1}: \quad v_{i+1} \leq v_i \quad \text{ and } \quad v_i \geq \frac{k}{2^{i+1}}.
		\end{equation}
		A direct calculation shows that $v_{i+1}$ is a solution to 
		\renewcommand{\v}{v_{i+1}}
		\begin{align*}
		\dot v_{i+1} + v_{i+1} \nabla \cdot \mathbf V_p - \delta_\Omega \Delta \v + \nabla\cdot (\mathbf J_\Omega u) &= 0 \quad \text{ in } \Omega(t)\\
		\delta_\Omega \nabla \v \cdot \nu - u j &= z \quad \text{ on } \Gamma(t)
		\end{align*}
		Reasoning as above, we aim to test the system above with $\v$. Note that 
		\begin{align*}
		\int_{\O(t)} \dot v_{i+1} \v + \int_{\O(t)} \v^2 \nabla \cdot \mathbf V_p &= \dfrac{1}{2}\dfrac{d}{dt}\int_{\O(t)} \v^2 + \dfrac{1}{2} \int_{\O(t)}\v^2 \nabla\cdot\mathbf V_p \\
		&=\dfrac{1}{2}\dfrac{d}{dt}\int_{\O(t)} \v^2 - \dfrac{1}{2}\int_{\O(t)} \nabla (\v^2) \cdot \mathbf V_p + \dfrac{1}{2}\intG \v ^2\mathbf V_p \cdot \nu \\
		&=\dfrac{1}{2}\dfrac{d}{dt}\int_{\O(t)} \v^2 - \dfrac{1}{2}\int_{\O(t)} \nabla (\v^2) \cdot \mathbf V_p + \dfrac{1}{2}\intG \v^2 \mathbf V_\Gamma \cdot \nu, 
		\end{align*}
		and integration by parts gives
		\begin{align*}
		-\delta_\Omega \int_{\O(t)} \Delta \v \, \v = \delta_\Omega \int_{\O(t)} |\nabla \v|^2 - \intG \v \nabla \v \cdot \nu.
		\end{align*}
		As for the term involving $\mathbf J_\Omega$, we have
		\begin{align*}
		\int_{\O(t)} \nabla \cdot (\mathbf J_\Omega u) \v &= - \int_{\O(t)} \mathbf J_\Omega u \cdot \nabla \v + \intG \v ju \\ 
		&= \int_{\O(t)} \mathbf V_p u \cdot \nabla \v - \int_{\O(t)} \mathbf V_\Omega u \cdot \nabla \v +  \intG \v ju \\
		&= \dfrac{1}{2}\int_{\O(t)} \mathbf V_p \cdot \nabla (\v^2) - \int_{\O(t)} \mathbf V_\Omega u \cdot \nabla \v +  \intG \v ju.
		\end{align*}		
		We finally combine all of the above to conclude that
		\begin{equation}\label{l7.1}
		\dfrac{1}{2}\dfrac{d}{dt} \int_{\O(t)} \v^2 + \delta_\Omega \int_{\O(t)} |\nabla \v|^2 = -\dfrac{1}{2}\intG \v^2 \mathbf V_\Gamma \cdot \nu - \int_{\O(t)} \mathbf V_\Omega u \cdot \nabla \v + \intG z \v.
		\end{equation}
		We now estimate all the terms on the right hand side. For the first and last terms, we have 
		\begin{align*}
		\left|\dfrac{1}{2}\intG \v^2 \mathbf V_\Gamma\cdot \nu \right| \leq \|\mathbf V_\Gamma\|_\infty \intG |\v|^2 \quad \text{ and } \quad \left|\intG z \v\right| \leq \intG |\v| \, |z|.
		\end{align*}
		To estimate the second term, we use the definition of $\v$ to write 
		\begin{align*}
		\left| \int_{\O(t)} \mathbf V_\Omega u \cdot \nabla \v \right| &= \left| \int_{\O(t)} \mathbf V_\Omega  \left(\v + k - k/2^{i+1}\right) \cdot \nabla \v  \right| \\
		&\leq \|\mathbf V\|_\infty \int_{\O(t)} |\v| \, |\nabla \v| + \left(k - k/2^{i+1}\right) \left| \int_{\O(t)} \VV_{\O}  \nabla \v\right| \\
		&\qquad + (k-k/2^{i+1})\left|\intG v_{i+1}\VV_{\O}\cdot \nu \right|\\
		&\leq \dfrac{\delta_\Omega}{2} \int_{\O(t)} |\nabla \v|^2 + C(T) \int_{\O(t)} |\v|^2 + C(k)\|\na\cdot \VV_{\O}\|_{\infty}\int_{\O(t)}|v_{i+1}|\\
		&\qquad + C\|\VV_{\O}\|_{\infty}\intG |v_{i+1}|
		\end{align*}
		
		Inserting the estimates above into \eqref{l7.1} and adding to both sides $\frac{\dO}{2}\int_{\O(t)}|v_{i+1}|^2 $ we get 
		\begin{equation*} 
		\begin{aligned}
		&\frac{d}{dt}\|v_{i+1}\|_{L^2(\O(t))}^2 + \frac{\dO}{2}\|v_{i+1}\|_{H^1(\O(t))}^2\\
		& \leq {C(T)}\left(\int_{\O(t)}(|v_{i+1}|^2 + |v_{i+1}|) + \int_{\GG(t)}|v_{i+1}|^2  + \int_{\GG(t)}|v_{i+1}|(|z|+1)\right).
		\end{aligned}
		\end{equation*}
		Integrating on $(0,T)$ and defining
		\begin{align*}
		Y_{i+1}:= \|v_{i+1}\|_V^2 = \sup_{t\in (0,T)}\|v_{i+1}(t)\|_{L^2(\O(t))}^2 + \int_0^T\|v_{i+1}\|_{H^1(\O(t))}^2 
		\end{align*}		
		yields 
		\begin{equation}\label{l7.4}
		\begin{aligned}
		Y_{i+1}\leq \|v_{i+1}(0)\|_{L^2(\O_0)}^2+ &{C(T)}\bigg(\int_0^T\int_{\O(t)}(|v_{i+1}|^2 + |v_{i+1}|)\\
		&\qquad + \int_0^T\int_{\GG(t)}|v_{i+1}|^2 + \int_0^T\int_{\GG(t)}|v_{i+1}|(|z|+1)\bigg).
		\end{aligned}
		\end{equation}
		We now estimate the terms on the right hand side of \eqref{l7.4}. First, by choosing $k \geq 2\|u_0\|_{L^\infty(\O_0)}$ we have 
		\begin{equation}\label{l7.4_1}
		\|v_{i+1}(0)\|_{L^2(\O_0)}^2 = \|(u_0 - k + k/2^{i+1})_+\|_{L^2(\O_0)}^2 = 0.
		\end{equation}
		Using \eqref{layer} we get
		\begin{equation}\label{l7.5}
		\begin{aligned}
		\int_0^T\int_{\O(t)}(|v_{i+1}|^2  + |v_{i+1}|)  &= \int_0^T\int_{\O(t)}(|v_{i+1}|^2  + |v_{i+1}|)\mathbf 1_{A_{i+1}}  \\
		&\overset{(\ref{layer})}{\leq} \int_0^T\int_{\O(t)}(|v_{i}|^2+|v_i|)\mathbf 1_{A_{i+1}}  \\
		&\overset{(\ref{layer})}{\leq}  \left[\left(\frac{2^{i+1}}{k}\right)^{4/d}+\left(\frac{2^{i+1}}{k}\right)^{(d+4)/(2d+4)}\right]\int_0^T\int_{\O(t)}|v_i|^{2+\frac{4}{d}}  \\
		&\leq C(k)B_1^{i}\|v_i\|_V^{2+\frac{4}{d}} = C(k)B_1^{i}Y_i^{1 + \frac{2}{d}}.
		\end{aligned}
		\end{equation}
		The boundary terms can be estimated in a similar way,
		\begin{equation}\label{l7.7}
		\begin{aligned}
		\int_0^T\int_{\GG(t)}|v_{i+1}|^2   &= \int_0^T\int_{\GG(t)}|v_{i+1}|^2\mathbf 1_{B_{i+1}}  \\
		&\leq \int_0^T\int_{\GG(t)}|v_{i}|^2\mathbf 1_{B_{i+1}}  \\
		&\leq \left(\frac{2^{i+1}}{k}\right)^{2/d}\int_0^T\int_{\GG(t)}|v_i|^{2+\frac 2d}  \\
		&\leq C(k)B_3^i\|v_i\|_V^{2+\frac 2d} = C(k)B_3^iY_i^{1+\frac 1d}.
		\end{aligned}
		\end{equation}
		Finally,
		\begin{equation}\label{l7.9}
		\begin{aligned}
		\int_0^T\int_{\GG(t)}|v_{i+1}|(|z|+1)   &\leq \int_0^T\int_{\GG(t)}\mathbf 1_{B_{i+1}}|v_i|(|z|+1)  \\
		&\leq \left(\frac{2^{i+1}}{k}\right)^{1+\frac{2}{d(d+2)}}\int_0^T\int_{\GG(t)}(|z|+1)|v_i|^{2+\frac{2}{d(d+2)}}  \\
		&\leq C(k)B_5^i\||z|+1\|_{L^{d+2}_{L^{d+2}(\GG)}}\|v_i\|_{L^{2+2/d}_{L^{2+2/d}(\GG)}}^{2+\frac{2}{d(d+2)}}\\
		&\overset{(\ref{inter})}{\leq} C(k)C(T)B_5^iY_i^{1+\frac{1}{d(d+2)}}.
		\end{aligned}
		\end{equation}
		Applying the estimates \eqref{l7.4_1}--\eqref{l7.9} into \eqref{l7.4} we get
		\begin{equation}\label{l7.9_1}
		Y_{i+1} \leq C(k)C(T)B^{i}\left(Y_i^{1+\frac 2d} + Y_i^{1+\frac 1d} + Y_i^{1 + \frac{1}{d(d+2)}}\right)
		\end{equation}
		with $B= \max\{B_1, \ldots, B_5\}$. To apply Moser's iteration, we need $Y_1$ to be small enough. From \eqref{l7.4} we have
		\begin{equation}\label{l7.10}
		\begin{aligned}
		Y_1 \leq \|v_{1}(0)\|_{L^2(\O_0)}^2&+ {C(T)}\int_0^T\int_{\O(t)}(|v_{1}|^2+|v_1|)+ {C(T)}\int_0^T\int_{\GG(t)}|v_{1}|^2 + C\int_0^T\int_{\GG(t)}|v_{1}|(|z|+1)
		\end{aligned}
		\end{equation}
		Choosing $k \geq 2\|u_0\|_{L^\infty(\O_0)}$ it yields  $\|v_1(0)\|_{L^2(\O_0)} = 0$. Using \eqref{layer} we have
		\begin{equation*}
		\begin{aligned}
		\int_0^T\int_{\O(t)}(|v_1|^2+|v_1|)\mathbf 1_{A_1} &\leq \left[\left(\frac 2k\right)^{4/d} + \left(\frac 2k\right)^{(d+4)/(2d+4)}\right]\int_0^T\int_{\O(t)}|u|^{2+\frac 4d}\\
		&\leq  \left[\left(\frac 2k\right)^{4/d} + \left(\frac 2k\right)^{(d+4)/(2d+4)}\right]\|u\|_V^{2+4/d},
		\end{aligned}
		\end{equation*}
		and
		\begin{equation*}
		\int_0^T\int_{\GG(t)}|v_1|^2\mathbf 1_{B_1}   \leq \left(\frac 2k \right)^{2/d}\|u\|_V^{2+2/d},
		\end{equation*}
		and finally, using $\||z|+1\|_{L^{d+2}_{L^{d+2}(\GG)}} \leq C(T)$,
		\begin{equation*}
		\int_0^T\int_{\GG(t)}|v_1|(|z|+1)\mathbf 1_{B_1}   \leq \left(\frac 2k\right)^{1+\frac{1}{d(d+2)}}C(T)\|u\|_V^{2+\frac{1}{d(d+2)}}.
		\end{equation*}
		Thus, from \eqref{l7.10} and the previous estimate that $\|u\|_V \leq C(T)$
		\begin{equation*}
		Y_1 \leq C(T)\left[\left(\frac 2k\right)^{4/d} + \left(\frac 2k\right)^{(d+4)/(2d+4)}+\left(\frac 2k\right)^{1+4/d} + \left(\frac 2k\right)^{2/d} + \left(\frac 2k\right)^{1+\frac{1}{d(d+2)}}\right].
		\end{equation*}
		For any $\varepsilon>0$, we can now choose $k$ large enough such that we obtain $Y_1 \leq \varepsilon$, and it thus follows from the Moser iteration \eqref{l7.9_1} that
		\begin{equation*}
		\lim_{i\to \infty}Y_i = 0,
		\end{equation*}
		or equivalently
		\begin{equation*}
		\|u\|_{L^\infty_{L^\infty(\O)}} \leq \frac k2 = C(T).
		\end{equation*}
	\end{proof}
	
	\begin{lemma}\label{L-infinity-manifold}
		Assume that $z\geq 0$ with $\|z\|_{L^r_{L^r(\GG)}} \leq C(T)$ for all $r\ge d+1$ and $w_0 \in L^\infty(\GG_0)$. Then the nonnegative solution to
		\begin{equation*}
		\begin{aligned}
		\dot{w} + w\naG\cdot \VV_p - \dG\Delta_\GG w + \naG\cdot(\J_\GG w) &= z, && \text{ on } \GGG_T,\\
		w(x,0) &= w_0(x), &&\text{ on } \GG_0,
		\end{aligned}
		\end{equation*}
		which satisfies $\|w\|_{L^\infty_{L^1(\GG)}} \leq M(T)$, is bounded in $L^\infty$, i.e.
		\begin{equation*}
		\|w\|_{L^\infty_{L^\infty(\GG)}} \leq C(T).
		\end{equation*}
	\end{lemma}
	\begin{proof}
		The proof of this lemma is similar to Lemma \ref{L-infinity-Neumann} and thus we omit it here.
	\end{proof}

	\subsection{Proof of Theorem \ref{thm:main}}
	To prove Theorem \ref{thm:main}, we still need to weak comparison principles which are shown in the next two lemmas.
	\begin{lemma}\label{comparison-surface}
		Let $y$ be a weak solution satisfying 
		\begin{equation*}
		\begin{aligned}
		\dot{y} + y \naG\cdot \VV_p + \naG\cdot(\J_\GG y) - \delta \Delta_\GG y  &=f, \quad \text{ on } \GGG_T,\\
		y(0) &= y_0 \leq 0, \quad \text{ on } \GG_0
		\end{aligned}
		\end{equation*}
		where $L^2_{L^2(\Gamma)} \ni f \le 0$ then $y(x,t) \leq 0$ a.e. on $\GGG_T$.
	\end{lemma}
	\begin{proof}
		Define $y_+ = \max\{0; y\} \geq 0$ and note that $y_+(0) = 0$. Multiplying the equation of $y$ by $y_+$ then integrating over $\GG(t)$ leads  to
		\begin{equation*}
		\int_{\GG(t)}[\dot{y} + y \naG\cdot \VV_p + \naG\cdot(\J_\GG y) - \delta \Delta_\GG y ]y_+    = \int_{\GG(t)}fy_+  \leq 0.
		\end{equation*}
		The following calculations are similar to those in Lemma \ref{L-infinity-Neumann}. We first use Lemma \ref{lem:derivinnerprod} and integration by parts to note that 
		\begin{align*}
		\intG \dot y y_+ = \dfrac{1}{2}\dfrac{d}{dt} \intG |y_+|^2 - \dfrac{1}{2} \intG y_+^2 \naG\cdot \mathbf V_p \quad \text{ and } \quad -\delta \intG (\Delta_\Gamma y) y_+ = \delta \intG |\naG y_+|^2,
		\end{align*}
		and then the divergence theorem to obtain
		\begin{align*}
		\intG y y_+ \naG \cdot \mathbf V_p = \intG y_+^2 \naG \cdot \mathbf V_p = \dfrac{1}{2} \intG y_+^2 \naG \cdot \mathbf V_p - \dfrac{1}{2} \intG \naG (y_+^2) \cdot \mathbf V_p
		\end{align*}
		and 
		\begin{align*}
		\intG \naG \cdot (\mathbf J_\Gamma y) y_+ &= - \intG (\mathbf V_\Gamma y) \cdot \naG y_+ + \intG \mathbf V_p y \cdot \naG y_+\\
		&= - \intG y_+ \mathbf V_\Gamma\cdot \naG y_+ + \dfrac{1}{2} \intG \mathbf V_p \cdot \naG (y_+)^2.
		\end{align*}
		Adding the four identities above leads  to
		\begin{equation}\label{l5:e1}
		\frac 12\pa_t\int_{\GG(t)}|y_+|^2  + \delta \int_{\GG(t)}|\naG y_+|^2  \leq \int_{\GG(t)}y_+ \VV_\GG \cdot \naG y_+ .
		\end{equation}
		By using Cauchy-Schwarz's inequality we have
		\begin{equation*}
		\begin{aligned}
		\int_{\GG(t)}y_+ \VV_\GG \cdot \naG y_+  &\leq \|\VV_\GG\|_{\infty}\int_{\GG(t)}|y_+||\naG y_+| \\
		&\leq \frac{\delta}{2}\int_{\GG(t)}|\naG y_+|^2  + \frac{\|\VV_\GG\|_\infty^2}{2\delta}\int_{\GG(t)}|y_+|^2 
		\end{aligned}
		\end{equation*}
		Inserting this into \eqref{l5:e1} we obtain
		\begin{equation*}
		\pa_t \int_{\GG(t)}|y_+|^2  \leq \frac{\|\VV_\GG\|_\infty^2}{\delta}\int_{\GG(t)}|y_+|^2 ,
		\end{equation*}
		and consequently, thanks to Gronwall's inequality,
		\begin{equation*}
		\int_{\GG(t)}|y_+(t)|^2  \leq \exp(t\|\VV_\GG\|_\infty^2/\delta)\int_{\GG_0}|y_+(0)|^2  = 0,
		\end{equation*}
		since $y_+(0) = 0$. This means that $y_+(x,t) = 0$ on $\GGG_T$ or, in other words, $y(x,t) \leq 0$ on $\GGG_T$.
	\end{proof}
	\begin{lemma}\label{comparison-Neumann}
		Let $u$ be a weak solution to 
		\begin{equation*}
		\begin{aligned}
		\dot{u} + u\na \cdot \VV_p - \dO\Delta u + \na\cdot(\J_\O u) &= f, \quad \text{ in } Q_T,\\
		\dO\na u \cdot \nu - uj &= g, \quad \text{ on } \GGG_T,\\
		u(0) &= u_0 \leq 0, &\quad \text{ in } \O_0,
		\end{aligned}
		\end{equation*}
		where $L^2_{L^2(\O)}\ni f \leq 0$ and $L^2_{L^2(\GG)} \ni g \leq 0$. Then $u(x,t) \leq 0$ in $Q_T$.
	\end{lemma}
	\begin{proof}
		The proof is similar to Lemma \ref{comparison-surface}. Let $u_+ = \max\{0; u\} \geq 0$ and note that $u_+(0) = 0$.
		Multiplying the equation of $u$ by $u_+$ then integrating over $\O(t)$ leads  to
		\begin{equation*}
		\int_{\O(t)}[\dot{u} + u\na \cdot \VV_p - \dO\Delta u + \na\cdot(\J_\O u)]u_+  = \int_{\O(t)} fu_+  \leq 0.
		\end{equation*}
		To simplify the expression above we perform calculations similar to those in the previous proof. On the one hand, we have
		\begin{align*}
		\int_{\O(t)} \dot u u_+ = \dfrac{1}{2} \dfrac{d}{dt} \int_{\O(t)} |u_+|^2 - \dfrac{1}{2}\int_{\O(t)} u_+^2 \nabla \cdot \mathbf V_p 
		\end{align*}
		and 
		\begin{align*}
		-\delta_\Omega \int_{\O(t)} (\Delta u) u_+ = \delta_\Omega \int_{\O(t)} |\nabla u_+|^2 - \delta_\Omega \intG u_+ \nabla u \cdot \nu.
		\end{align*}
		On the other hand, for the terms involving the velocities, we obtain
		\begin{align*}
		\int_{\O(t)} u u_+ \nabla \cdot \mathbf V_p = \dfrac{1}{2}\int_{\O(t)} u_+^2 \nabla \cdot \mathbf V_p - \dfrac{1}{2} \int_{\O(t)} \nabla (u_+^2) \cdot \mathbf V_p + \dfrac{1}{2}\intG u_+^2 \mathbf V_p \cdot \nu.
		\end{align*}
		and
		\begin{align*}
		\int_{\O(t)} \nabla \cdot (\mathbf J_\Omega u) u_+ &= -\int_{\O(t)} u_+ \mathbf V_\Omega \cdot \nabla u + \int_{\O(t)} u_+ \mathbf V_p \cdot \nabla u + \intG u \, u_+\, \mathbf J_\Omega \cdot \nu \\
		&=\int_{\O(t)} u_+ \mathbf V_\Omega \cdot \nabla u +\dfrac{1}{2} \int_{\O(t)} \nabla(u_+^2) \cdot \mathbf V_p + \intG u \, u_+ \, j.
		\end{align*}
		Adding the four identities above and recalling the boundary condition $\delta_\Omega \nabla u\cdot  \nu - uj = g \leq 0$ leads  to
		\begin{align*}
		\dfrac{1}{2} \dfrac{d}{dt} \int_{\O(t)} |u_+|^2 + \delta_\Omega \int_{\O(t)} |\nabla u_+|^2 &\leq - \dfrac{1}{2} \intG u_+^2 \mathbf V_p \cdot \nu - \dfrac{1}{2} \int_{\O(t)} u_+ \mathbf V_\Omega \cdot \nabla u + \intG u_+ g \\
		&\leq - \dfrac{1}{2} \intG u_+^2 \mathbf V_p \cdot \nu - \dfrac{1}{2} \int_{\O(t)} u_+ \mathbf V_\Omega \cdot \nabla u \\
		&= - \dfrac{1}{2} \intG u_+^2 \mathbf V_\Gamma \cdot \nu - \dfrac{1}{2} \int_{\O(t)} u_+ \mathbf V_\Omega \cdot \nabla u.
		\end{align*}
		%
		%
		%
		Therefore we have
		\begin{equation*}
		\frac 12 \pa_t\int_{\O(t)}|u_+|^2  + \dO\int_{\O(t)}|\na u_+|^2  \leq \frac{\|\VV_\Gamma\|_\infty}{2} \int_{\GG(t)}|u_+|^2  + \|\VV_\O\|_{\infty}\int_{\O(t)}|u_+||\na u_+| .
		\end{equation*}
		Using the modified Trace inequality (see e.g. \cite[Theorem 1.5.1.10]{grisvard2011elliptic})
		\begin{equation*}
		\int_{\GG(t)}|u_+|^2  \leq \varepsilon \int_{\O(t)}|\na u_+|^2  + C_\varepsilon \int_{\O(t)}|u_+|^2 
		\end{equation*}
		and the Cauchy-Schwarz inequality
		\begin{equation*}
		\|\VV_\O\|_{\infty}\int_{\O(t)}|u_+||\na u_+|  \leq \varepsilon \int_{\O(t)}|\nabla u_+|^2  + C_\varepsilon\int_{\O(t)}|u_+|^2 ,
		\end{equation*}
		we get, by choosing $\varepsilon$ small enough, 
		\begin{equation*}
		\pa_t\int_{\O(t)}|u_+|^2  \leq C\int_{\O(t)}|u_+|^2 
		\end{equation*}
		and consequently $u_+(t) = 0$ since $u_+(0) = 0.$
	\end{proof}
	
	We are now ready to prove Theorem \ref{thm:main}.	
	\begin{proof}[Proof of Theorem \ref{thm:main}]
		We consider the case $f_1(u,w,z) + f_2(u,w,z) \leq 0$ in assumption \eqref{A3} and $f_2(u,w,z) \leq C(w^{\beta}+z^{\beta}+1)$ in assumption \eqref{A4} since the other cases are treated similarly. First of all, the assumption \eqref{A1} implies the nonnegativity of $u, w, z$ thanks to nonnegative initial data. By testing the equations of $u$, $w$ and $z$ with identity and integrating we easily obtain the following dissipation of masses 
		\begin{equation*}
		\pa_t\left[\int_{\O(t)}u  + \int_{\GG(t)}w  \right] \leq 0 \quad \text{ and } \quad \pa_t\left[\int_{\GG(t)}w  + \int_{\GG(t)}z  \right] \leq 0.
		\end{equation*}
		Combining these with the nonnegativity of solutions, we get the uniform bound
		\begin{equation}\label{L1-bound}
		\|u\|_{L^\infty_{L^1(\O)}} + \|w\|_{L^\infty_{L^1(\GG)}}+ \|z\|_{L^\infty_{L^1(\GG)}} \leq M
		\end{equation}
		where $M = \max\{\|u_0\|_{L^1(\Omega_0)} + \|w_0\|_{L^1(\GG_0)}; \|w_0\|_{L^1(\GG_0)} + \|z_0\|_{L^1(\GG_0)} \}$.	By summing the equation of $w$ and $z$ we get \eqref{sum_wz}, and therefore by Lemma \ref{duality-L2}
		\begin{equation*}
		\|w\|_{L^2_{L^2(\Gamma)}} + \|z\|_{L^2_{L^2(\Gamma)}} \leq C(T).
		\end{equation*}
		Moreover, from $f_2(u,w,z) \leq C(w^{\beta} + z^{\beta}+1)$ with $\beta < \frac{d+4}{d+2} < 2$, it follows from comparison principle in Lemma \ref{comparison-surface} that
		\begin{equation*}
		\|w\|_{L^{s_0}_{L^{s_0}(\Gamma)}} + \|z\|_{L^{s_0}_{L^{s_0}(\Gamma)}} \leq C(T)
		\end{equation*}
		for $s_0 = \frac{2}{\beta} >1$. We use this estimate and apply Lemmas  \ref{heat-regularity}, \ref{comparison-surface} to the equation for $w$ and recall that $f_2(u,w,z)\leq C(w^\beta + z^\beta + 1)$ to get
		\begin{equation*}
		\|w\|_{L^{s_1}_{L^{s_1}(\GG)}} \leq C(T)  \quad \text{ where } \quad s_1 = \begin{cases}
		<\infty &\text{ arbitrary if } s_0 \geq (d+2)/2,\\
		< \frac{(d+2)s_0}{\beta(d+2-s_0)} &\text{ arbitrary if } s_0 < (d+2)/2.
		\end{cases}
		\end{equation*}
		From Lemma \ref{duality-1} it follows
		\begin{equation*}
		\|z\|_{L^{s_1}_{L^{s_1}(\GG)}} \leq C(T).
		\end{equation*}
		Repeating this procedure we construct a sequence $\{s_n\}$ such that
		\begin{equation*}
		s_{n+1} = \begin{cases}
		<\infty &\text{ arbitrary if } s_n \geq (d+2)/2,\\
		< \frac{(d+2)s_n}{\beta(d+2-s_n)} &\text{ arbitrary if } s_n < (d+2)/2.
		\end{cases}
		\end{equation*}
		Since $s_0 = \frac{2}{\beta} > \frac{2(d+2)}{d+4}$, it is easy to see that $s_n$ is strictly increasing, and since $s_{n+1}/s_n > 1$, after some finite iteration we get $s_n \geq (d+2)/2$, and therefore $s_{n+1} < +\infty$ arbitrary. That means we have
		\begin{equation}\label{wz-Lr}
		\|w\|_{L^{s}_{L^{s}(\GG)}} + \|z\|_{L^{s}_{L^{s}(\GG)}} \leq C(T) \quad \text{ for all } \quad s\in [1,\infty).
		\end{equation}
		From this and Lemmas \ref{comparison-surface}, \ref{L-infinity-manifold} we obtain the $L^\infty$ estimate 
		\begin{equation}\label{Linf-w}
		\|w\|_{L^\infty_{L^\infty(\GG)}} \leq C(T).
		\end{equation}
		Consider now the equation 
		\newcommand{\bu}{\overline{u}}
		\begin{equation*}
		\left\{
		\begin{aligned}
		\dot{\bu} + \bu \na \cdot \VV_p - \dO \Delta \bu + \na\cdot(\J_\O \bu) &= 0, &&\text{ in } Q_T,\\
		\dO \na\bu \cdot \nu - \bu j &= C(w^{\alpha} + z^{\alpha} + 1), &&\text{ on } \GGG_T,\\
		\bu(x,0) &=u_0(x), &&\text{ in } \O_0,
		\end{aligned}\right.
		\end{equation*}
		and note $\|w^{\alpha} + z^{\alpha} + 1\|_{L^r_{L^r(\GG)}} \leq C(T)(\|w\|_{L^{\alpha r}_{L^{\alpha r}(\GG)}}^{\alpha} + \|z\|_{L^{\alpha r}_{L^{\alpha r}(\GG)}}^{\alpha} + 1) \leq C(T)$ for any $r\in [1,\infty)$. We can apply Lemma \ref{L-infinity-Neumann} to have
		\begin{equation*}
		\|\bu\|_{L^\infty_{L^\infty(\O)}} \leq C(T),
		\end{equation*}
		and then by assumption \eqref{A4_1} and the comparison in Lemma \ref{comparison-Neumann} it follows 
		\begin{equation}\label{Linf-u}
		\|u\|_{L^\infty_{L^\infty(\O)}} \leq C(T).
		\end{equation}
		In order to prove the $L^\infty$-bound of $z$, we first show that
		\begin{equation*}
		\|u\|_{L^2_{H^1(\O)}} \leq C(T).
		\end{equation*}
		We multiply the equation for $u$ by $u$ and integrate over $\O(t)$. Note that this is the same equation as in Lemma \ref{comparison-Neumann} with $f=0$ and $g=f_1$. Reusing the calculation in the proof of Lemma \ref{comparison-Neumann} we then obtain
		\begin{align}\label{t:e1}
		\dfrac{1}{2} \dfrac{d}{dt} \int_{\O(t)} |u|^2 + \delta_\Omega \int_{\O(t)} |\nabla u|^2 &= - \dfrac{1}{2} \intG u^2 \mathbf V_\Gamma \cdot \nu + \intG u f_1(u,w,z) - \dfrac{1}{2} \int_{\O(t)} u \mathbf V_\Omega \cdot \nabla u  \\
		&=: I_1 + I_2 + I_3,
		\end{align}
		We estimate the terms on the last right hand side of \eqref{t:e1}. 
		\begin{equation}\label{I1}
		|I_1| \leq \frac 12 \intG |u|^2|\VV_\Gamma\cdot \nu|  \leq \|\VV\|_\infty \, \|u\|_{L^2(\GG(t))}^2 \leq {C(T)}\|u\|_{L^2(\GG(t))}^2,
		\end{equation}
		due to $(\VV_\Gamma|_{\GG}\cdot \nu)\nu = \VV$. For $I_2$ we use \eqref{A4_1} to estimate
		\begin{equation}\label{I2}
		\begin{aligned}
		I_2 &\leq C\intG u(w^{\alpha} + z^{\alpha} + 1) \\
		& \leq C(T)(1 + \|w\|_{L^{2\alpha}(\GG(t))}^{2\alpha} + \|z\|_{L^{2\alpha}(\GG(t))}^{2\alpha} + \|u\|_{L^2(\GG(t))}^2)
		\end{aligned}
		\end{equation}
		thanks to \eqref{Linf-w} and \eqref{wz-Lr}. With $I_3$ we use Cauchy-Schwarz's inequality
		\begin{equation}\label{I3}
		|I_3| \leq {\|\VV_\O\|_{\infty}}\int_{\O(t)}|u||\na u|  \leq {C(T)}\|u\|_{L^2(\O(t))}^2 + \frac{\dO}{2}\int_{\O(t)}|\na u|^2  \leq C(T) + \frac{\dO}{2}\int_{\O(t)}|\na u|^2 ,
		\end{equation}
		due to \eqref{Linf-u}. Using \eqref{I1}, \eqref{I2} and \eqref{I3} into \eqref{t:e1} we get
		\begin{equation*}
		\pa_t\|u\|_{L^2(\O(t))}^2 + \int_{\O(t)}|\na u|^2  \leq C(T)\left(1 + \|w\|_{L^{2\alpha}(\GG(t))}^{2\alpha} + \|z\|_{L^{2\alpha}(\GG(t))}^{2\alpha} + \|u\|_{L^2(\GG(t))}^2\right).
		\end{equation*}
		Adding both sides with $\|u\|_{L^2(\O(t))}^2$ then integrating on $(0,T)$ yields 
		\begin{equation}\label{t:e2}
		\begin{aligned}
		\int_0^T\|u\|_{H^1(\O(t))}^2  &\leq \|u_0\|_{L^2(\O_0)}^2+ C(T)\left[T + \|u\|_{L^2_{L^2(\O)}}^2 + \|u\|_{L^2_{L^2(\GG)}}^2 + \|w\|_{L^{2\alpha}_{L^{2\alpha}(\GG)}}^{2\alpha} + \|z\|_{L^{2\alpha}_{L^{2\alpha}(\GG)}}^{2\alpha}\right]\\
		&\leq C(T)\left[1+\int_0^T\|u\|_{L^2(\GG(t))}^2  \right]
		\end{aligned}
		\end{equation}
		using \eqref{wz-Lr} and \eqref{Linf-u}. Now we use the {interpolation} inequality
		\begin{equation*}
		\|u\|_{L^2(\GG(t))}^2 \leq {C(T)}\|u\|_{H^1(\O(t))}\|u\|_{L^2(\O(t))}
		\end{equation*}
		to estimate
		\begin{equation*}
		\begin{aligned}
		C(T)\int_0^T\|u\|_{L^2(\GG(t))}^2  &\leq C(T)\int_0^T\|u\|_{H^1(\O(t))}\|u\|_{L^2(\O(t))} \\
		&\leq \frac 12 \int_0^T\|u\|_{H^1(\O(t))}^2  + \frac{C(T)^2}{2}\int_0^T\|u\|_{L^2(\O(t))}^2 \\
		&\leq \frac 12 \int_0^T\|u\|_{H^1(\O(t))}^2  + C(T)
		\end{aligned}
		\end{equation*}
		where we used \eqref{Linf-u}. Putting this estimate into \eqref{t:e2} we get the desired estimate
		\begin{equation}\label{H1-u}
		\int_0^T\|u\|_{H^1(\O(t))}^2  \leq C(T).
		\end{equation}
		Now for any $r\geq 2$ we write $r = 2\lambda$ and estimate
		\begin{equation*}
		\begin{aligned}
		\|u\|_{L^r_{L^r(\GG)}}^r &= \int_{\GGG_T}|u^\lambda|^2 \\
		&\leq {C(T)}\int_{Q_T}\left[|\na(u^\lambda)|^2 + |u^\lambda|^2\right]\dV\\
		&= {C(T)}\int_{Q_T}[\lambda^2|u|^{2(\lambda-1)}|\nabla u|^2 + |u|^{2\lambda}]\dV\\
		&\leq {C(T)}\lambda^2\|u\|_{L^\infty_{L^\infty(\O)}}^{2(\lambda-1)}\int_0^T\|u\|_{H^1(\O(t))}^2  + {C(T)}\|u\|_{L^{2\lambda}_{L^{2\lambda}(\O)}}^{2\lambda}\\
		&\leq C(T)
		\end{aligned}
		\end{equation*}
		using \eqref{Linf-u} and \eqref{H1-u}. Recall that we have $\|u\|_{L^r_{L^r(\GG)}}, \|w\|_{L^r_{L^r(\GG)}}, \|z\|_{L^r_{L^r(\GG)}} \leq C(T)$ for all $r\in [1,\infty)$. Now using the assumption \eqref{A5} that $f_3(u,w,z)$ is at most polynomial in $u, w$ and $z$, we imply that $\|f_3(u,w,z)\|_{L^r_{L^r(\GG)}} \leq C(T)$ for all $r\in [1,\infty)$. Therefore, applying Lemma \ref{L-infinity-manifold} to the equation of $z$ leads  finally to
		\begin{equation}\label{Linf-z}
		\|z\|_{L^\infty_{L^\infty(\GG)}} \leq C(T).
		\end{equation}
	\end{proof}
	
	\section{Convergence to equilibrium}\label{sec:convergence}
	As mentioned in the introduction, the convergence to equilibrium of \eqref{sys} stated in Theorem \ref{thm:convergence} follows from a vector-valued functional inequality, which can also be of independent interest. This section is therefore split into two subsections, where in the first one we show the functional inequality \eqref{e1} while in the second one we utilise this inequality to prove Theorem \ref{thm:convergence}.
	
	\subsection{Vector-valued functional inequalities}\label{subs:entropy_inequality}
	The convergence to equilibrium relies essentially on the following entropy-entropy dissipation estimate. It is again emphasised that this is a pure vector-valued functional inequality, which does not use any properties of the system \eqref{sys}.
	We are going to utilise the following inequalities for a bounded domain with smooth boundary $\Gamma = \pa\Omega$,
	\begin{itemize}
		\item the Log-Sobolev-Inequality, see e.g. \cite{arnold2001convex},
		\begin{equation}\label{LSI}
		\intOO \frac{|\na f|^2}{f} \ge C_{\text{LSI,$\Omega$}}\intOO f\log \frac{f}{\overline f}, \quad \intGG \frac{|\na_\GG g|^2}{g} \ge C_{\text{LSI,$\GG$}}\intGG g\log\frac{g}{\overline g},
		\end{equation}	
		\item the Trace-Poincar\'e-Wirtinger inequality, see e.g. \cite{grisvard2011elliptic},
		\begin{equation}\label{TrPW}
		\intOO|\na f|^2 \ge C_{\text{TrPW,$\Omega$}}\intGG |f - \overline f|^2,
		\end{equation}
		\item and the Poincar\'e-Wirtinger inequality, see e.g. \cite{taylor2013partial},
		\begin{equation}\label{PW}
		\intGG |\na_\GG g|^2 \ge C_{\text{PW,$\GG$}}\intGG |g - \overline g|^2
		\end{equation}
		where the constants do not depend on $f$ and $g$, and recalling
		\begin{equation*}
		\overline f = \frac{1}{\measv(\Omega)}\intOO f(x), \quad \overline g = \frac{1}{\measb(\GG)}\intGG g(x).
		\end{equation*}	 
	\end{itemize}
	We define, for given $(\uinf,\zinf,\winf)\in (0,\infty)^3$,
	\begin{equation}\label{def_entropy}
	E[u,w,z] = \intOO u\log\frac{u}{\uinf} - u + \uinf   + \intGG w\log\frac{w}{\winf}-w+\winf   + \intGG z\log\frac{z}{\zinf}-z+\zinf
	\end{equation}
	and
	\begin{equation}\label{def_ep}
	\wt{D}[u,w,z] = \intOO\frac{|\na u|^2}{u}  + \intGG\frac{|\naG w|^2}{w}  + \intGG\frac{|\naG z|^2}{z}  + \intGG(z-uw)\log\frac{z}{uw} .
	\end{equation}
	\begin{proposition}\label{pro:eed}
		Let $\Omega \subset \mathbb R^n$ be a bounded domain with smooth boundary $\Gamma = \partial\Omega$, e.g. $\Gamma$ is of class $C^{2+\alpha}$ for some $\alpha > 0$. Let $(\uinf, \zinf, \winf)\in (0,\infty)^3$ satisfy
		\begin{equation*}
		\zinf = \uinf\winf. 
		\end{equation*}
		Let $E[u,w,z]$ and $D[u,w,z]$ be defined as in \eqref{def_entropy} and \eqref{def_ep}. There exists a constant $\lambda >0$ depending only the parameters $\dO, \dG, \dGG$, the triple $(\uinf,\winf,\zinf)$, and the constants $C_{{\normalfont LSI},\Omega}$, $C_{\text{LSI},\Gamma}$, $C_{\text{TrPW},\Omega}$, $C_{\text{PW},\Gamma}$ in \eqref{LSI}, \eqref{TrPW}, \eqref{PW} such that
		\begin{equation*}
		\wt{D}[u,w,z] \geq \lambda E[u,w,z]
		\end{equation*}
		for all functions $u, w, z$ satisfying the conservation laws
		\begin{equation}\label{cons1}
		\int_{\Omega}u(x) + \int_{\Gamma}z(x) = \measv(\Omega)\uinf + \measb(\Gamma)\zinf,
		\end{equation}
		\begin{equation}\label{cons2}
		\int_{\Gamma}w(x) + \int_{\Gamma}z(x) = \measb(\Gamma)\winf + \measb(\Gamma)\zinf.
		\end{equation}
	\end{proposition}
	\begin{proof}
		It is noted that all constants in this lemma depend on the domains $\Omega$ and $\Gamma$ only through their measures $\measv(\Omega)$ and $\measb(\Gamma)$. Using the Logarithmic-Sobolev-Inequality we have 
		\begin{equation*}
		\frac{\dO}{4}\intOO\frac{|\na u|^2}{u}  \geq C_1\intOO u\log\frac{u}{\ou} , \quad \frac{\dG}{4}\intGG\frac{|\naG w|^2}{w}  \geq C_2\intGG w\log\frac{w}{\ow} , 
		\end{equation*}
		and
		\begin{equation*}
		\frac{\dGG}{4}\intGG \frac{|\naG z|^2}{z}  \geq C_3\intGG z\log\frac{z}{\oz},
		\end{equation*}
		where
		\begin{equation*}
		C_1 = \frac{\delta_\Omega}{4}C_{\text{LSI},\Omega}, \quad C_2 = \frac{\dG}{4}C_{\text{LSI},\GG}, \quad C_3 = \frac{\dGG}{4}C_{\text{LSI},\GG},
		\end{equation*}
		and
		\begin{equation*}
		\ou = \frac{1}{\measv(\Omega)}\intOO u(x),\quad \oz = \frac{1}{\measb(\Gamma)}\intGG z(x), \quad  \ow = \frac{1}{\measb(\Gamma)}\intGG w(x).
		\end{equation*}
		By direct computations, we have
		\begin{equation}\label{eb1}
		\begin{aligned}
		E[u,w,z] &= \intOO u\log\frac{u}{\ou}  + \intOO u\log\frac{\ou}{\uinf} - u + \uinf \\
		&\quad + \intGG w\log\frac{w}{\ow}  + \intGG w\log\frac{\ow}{\winf} - w + \winf  \\
		&\quad + \intGG z\log\frac{z}{\oz}  + \intGG z\log\frac{\oz}{\zinf} - z + \zinf  ,
		\end{aligned}
		\end{equation}
		so that comparing the expressions above we have
		\begin{align}\label{eq:dtilde2}
		\wt{D}[u,w,z] \geq C(E[u,w,z] - \overline{E}[u,w,z]),
		\end{align}
		where we define the remaining term
		\begin{align*}
		\overline{E}[u,w,z]:= \intOO u\log\frac{\ou}{\uinf} - u+\uinf   + \intGG w\log\frac{\ow}{\winf}-w+\winf  + \intGG z\log\frac{\oz}{\zinf}-z+\zinf .
		\end{align*}
		It then remains to control $\overline E[u,w,z]$. 				
		The rest of the proof is divided in several steps. Before going further, we introduce for convenience the square root notation
		\begin{equation*}
		U = \sqrt{u}, \quad W = \sqrt{w}, \quad Z = \sqrt{z}, \quad U_\infty = \sqrt{\uinf}, \quad W_\infty = \sqrt{\winf}, \quad Z_\infty = \sqrt{\zinf}
		\end{equation*}
		and their respective averages
		\begin{equation*}
		\overline{U} = \frac{1}{\measv(\Omega)|}\intOO U  , \quad \overline{W}=\frac{1}{\measb(\GG)|}\intGG W , \quad \overline{Z} = \frac{1}{\measb(\GG)}\intGG Z .
		\end{equation*}
		Because of the conservation laws \eqref{cons1} and \eqref{cons2}, there are $M_1, M_2>0$ depending only on $\uinf, \winf, \zinf$ such that
		\begin{equation*}
		\overline u + \overline z \le M_1, \quad \overline{w} + \overline{z} \le M_2.
		\end{equation*}
		
		\noindent\underline{Step 1.} We show that there exists a constant $C>0$ such that
		\begin{equation}\label{eb2}
		\overline{E}[u,w,z] \leq C\left(|\sqrt{\overline{U^2}} - U_\infty|^2 + |\sqrt{\overline{W^2}} - W_\infty|^2 + |\sqrt{\overline{Z^2}} - Z_\infty|^2\right).
		\end{equation} 
		We consider the function
		\begin{equation*}
		\Phi(x,y) = \frac{x\log(x/y) - x + y}{(\sqrt x - \sqrt y)^2}
		\end{equation*}
		which is continuous on $(0,\infty)^2$ and $\Phi(\cdot, y)$ is increasing for any fixed $y>0$. From that, we have
		\begin{align*}
		\intOO u\log\frac{\ou}{\uinf}-u+\uinf   &=\measv(\O)\Phi(\ou,\uinf)(\sqrt{\overline{U^2}} - U_\infty)^2\\
		&\leq \measv(\O)\Phi\left(\frac{M_1}{\measv(\O)}, \uinf\right)(\sqrt{\overline{U^2}}-U_\infty)^2
		\end{align*}
		and similarly
		\begin{align*}
		\intGG w\log\frac{\ow}{\winf}-w+\winf   \leq C(M_2, \measb(\Gamma))(\sqrt{\overline{W^2}}-W_\infty)^2,
		\end{align*}
		and
		\begin{align*}
		\intGG z\log\frac{\oz}{\zinf}-z+\zinf  \leq C(M_2,\measb(\Gamma) )(\sqrt{\overline{Z^2}}-Z_\infty)^2.
		\end{align*}
		Hence \eqref{eb2} is proved.
		
		\medskip
		\noindent\underline{Step 2.} We next prove
		\begin{equation}\label{eb3}
		\wt{D}[u,w,z] \geq C\left(\overline{Z} - \overline{U}\,\overline{W}\right)^2.
		\end{equation}
		For convenience, we introduction the deviation to averages,
		\begin{equation*}
		\delta_U(x) = U(x) - \overline{U}, \quad \delta_W(x) = W(x) - \overline{W}, \quad \delta_Z(x) = Z(x) - \overline{Z}.
		\end{equation*}
		By using $\frac{|\na u|^2}{u} = 4|\na U|^2$ and $\int_{\O(t)}(U - \overline{U})  = 0$, we can apply Trace-Poincar\'e-Wirtinger's inequality \eqref{TrPW} to have
		\begin{equation*}
		\intOO \frac{|\na u|^2}{u}  = 4\|\na U\|_{L^2(\O)}^2 \geq 4C_{\text{TrPW},\Omega}\|U - \overline{U}\|_{L^2(\GG)}^2 = 4C_{\text{TrPW},\Omega}\|\delta_U\|_{L^2(\GG)}^2.
		\end{equation*}
		The Poincar\'e-Wirtinger inequality on $\Gamma$ in \eqref{PW} gives
		\begin{equation*}
		\intGG\frac{|\naG w|^2}{w}  \geq C_{\text{PW},\GG}\|\delta_W\|_{L^2(\GG)}^2, \quad \text{and} \quad \intGG \frac{|\naG z|^2}{z}  \geq C_{\text{PW},\GG}\|\delta_Z\|_{L^2(\GG))}^2.
		\end{equation*}
		Combining these estimates with the elementary inequality $(x-y)\log(x/y) \geq (\sqrt x - \sqrt y)^2$, we obtain
		\begin{equation*}
		\wt{D}[u,w,z] \geq C(\|\delta_U\|_{L^2(\GG)}^2+\|\delta_W\|_{L^2(\GG)}^2 + \|\delta_Z\|_{L^2(\GG)}^2 + \|Z - UW\|_{L^2(\GG)}^2).
		\end{equation*}
		Fix a constant $L>0$ and define the set
		\begin{equation*}
		S_L = \{x\in\GG: |\delta_U(x)|, |\delta_W(x)|, |\delta_Z(x)| \leq L \} \quad \text{ and } \quad S_L^\top = \GG\backslash S_L.
		\end{equation*}
		In the following we will use the bounds
		\begin{align*}
		\overline{U} &= \measv(\O)^{-1}\int_{\O}U  \leq \sqrt{\measv(\O)}^{-1}\left(\int_{\Omega}U^2 \right)^{1/2}\\
		&= \sqrt{\measv(\O)}^{-1}(\sqrt{\measv(\O)}\,\ou)^{1/2} \leq \sqrt{M_1},
		\end{align*}
		\begin{equation*}
		\overline{W}\leq \sqrt{M_2}, \qquad \overline{Z} \leq \sqrt{M_1}.
		\end{equation*}
		We first estimate using the elementary inequality $(x-y)^2 \geq \frac 12 x^2 - y^2$,
		\begin{equation}\label{eb6}
		\begin{aligned}
		&\|Z - UW\|_{L^2(\GG)}^2 \geq\int_{S_L}|(\overline{Z}+\delta_Z) - (\overline{U}+\delta_U)(\overline{W} + \delta_W)|^2 \\
		&\geq \frac 12\int_{S_L}|\overline{Z} - \overline{U}\,\overline{W}|^2  - \int_{S_L}|\delta_Z - \delta_U(\overline{W} + \delta_W) - \delta_W \overline{U}|^2 \\
		&\geq \frac 12\int_{S_L}|\overline{Z} - \overline{U}\,\overline{W}|^2  - C(L,M_1,M_2)\int_{S_L}(|\delta_U|^2 + |\delta_W|^2 + |\delta_Z|^2) \\
		&\geq \frac 12\int_{S_L}|\overline{Z} - \overline{U}\,\overline{W}|^2  - C(L,M_1,M_2)(\|\delta_U\|_{L^2(\GG)}^2+\|\delta_W\|_{L^2(\GG)}^2+\|\delta_Z\|_{L^2(\GG)}^2).
		\end{aligned}
		\end{equation}
		On the other hand, using the fact that $|\delta_U|+|\delta_W|+|\delta_Z| \geq L$ on $S_L^\top$ we have
		\begin{equation}\label{eb7}
		\begin{aligned}
		\|\delta_U\|_{L^2(\GG))}^2 + \|\delta_W\|_{L^2(\GG)}^2 + \|\delta_Z\|_{L^2(\GG)}^2&\geq \frac 13\int_{S_L^\top}(|\delta_U| + |\delta_W| + |\delta_Z|)^2 \\
		&\geq \frac{L^2}{3}|S_L^\top|\\
		&\geq \frac{L^2}{3C(M_1,M_2)}\int_{S_L^\top}|\overline{Z}-\overline{U}\,\overline{W}|^2 
		\end{aligned}
		\end{equation}
		where we used $|\overline{Z}-\overline{U}\overline{W}|^2 \leq C(M_1,M_2)$ at the last step.
		
		Combining \eqref{eq:dtilde2}, \eqref{eb6} and \eqref{eb7} we have, for all $\theta \in(0,1)$,
		\begin{align*}
		\wt{D}[u,w,z]&\geq C(\|\delta_U\|_{L^2(\GG)}^2 + \|\delta_W\|_{L^2(\GG)}^2 + \|\delta_Z\|_{L^2(\GG)}^2 + \theta\|Z - UW\|_{L^2(\GG)}^2)\\
		&\geq C(L,M_1,M_2)\int_{S_L^\top}|\overline{Z}- \overline{U}\,\overline{W}|^2  + C\int_{S_L}|\overline{Z}-\overline{U}\,\overline{W}|^2 \\
		&\quad + C(1-C\theta)(\|\delta_U\|_{L^2(\GG)}^2+\|\delta_W\|_{L^2(\GG)}^2+\|\delta_Z\|_{L^2(\GG)}^2)\\
		&\geq C(L,M_1,M_2)\measb(\GG)(\overline{Z} - \overline{U}\,\overline{W})^2
		\end{align*}
		where we used $|S_L| + |S_L^\top| = \measb(\GG)$ and chose $\theta$ small enough. This finishes {Step 2.}
		
		\medskip
		\noindent\underline{Step 3.} There exists a constant $C>0$ such that
		\begin{equation}\label{eb4}
		\wt{D}[u,w,z] \geq C\left(\sqrt{\overline{Z^2}}-\sqrt{\overline{U^2}}\sqrt{\overline{W^2}} \right)^2.
		\end{equation}
		From {Step 2}, it remains to show that
		\begin{equation*}
		\|\delta_U\|_{L^2(\O)}^2 + \|\delta_W\|_{L^2(\GG)}^2 + \|\delta_Z\|_{L^2(\GG)}^2 + (\overline{Z}-\overline{U}\,\overline{W})^2 \geq C\left(\sqrt{\overline{Z^2}} - \sqrt{\overline{U^2}}\sqrt{\overline{W^2}}\right).
		\end{equation*}
		By definition
		\begin{equation*}
		\|\delta_U\|_{L^2(\O)}^2 = \measv(\Omega)(\overline{U^2}- \overline{U}^2) = \measv(\Omega)(\sqrt{\overline{U^2}} - \overline{U})(\sqrt{\overline{U^2}} + \overline{U}),
		\end{equation*}
		which implies
		\begin{equation*}
		\overline{U} = \sqrt{\overline{U^2}} - \frac{1}{ \measv(\Omega)}\frac{\|\delta_U\|_{L^2(\O)}^2}{\sqrt{\overline{U^2}} + \overline{U}} = \sqrt{\overline{U^2}} - R_U\|\delta_U\|_{L^2(\O)}
		\end{equation*}
		where
		\begin{equation*}
		R_U = \frac{1}{ \measv(\Omega)}\frac{\|\delta_U\|_{L^2(\O)}}{\sqrt{\overline{U^2}}+\overline{U}}.
		\end{equation*}
		We estimate
		\begin{equation*}
		|R_U|^2=\frac{1}{ \measv(\Omega)^2}\frac{\|\delta_U\|_{L^2(\O)}^2}{(\sqrt{\overline{U^2}}+\overline{U})^2} = \frac{1}{ \measv(\Omega)}\frac{{\overline{U^2}}-\overline{U}^2}{(\sqrt{\overline{U^2}}+\overline{U})^2} \leq \frac{1}{ \measv(\Omega)}.
		\end{equation*}
		Similarly we have
		\begin{equation*}
		\overline{W} = \sqrt{\overline{W^2}} - R_W\|\delta_W\|_{L^2(\GG)}, \quad \text{ and } \quad \overline{Z}= \sqrt{\overline{Z^2}} - R_Z\|\delta_Z\|_{L^2(\GG)}
		\end{equation*}
		with $0\leq R_W, R_Z \leq \frac{1}{\measb(\GG)}$. Using the mass conservation it yields 
		\begin{equation*}
		\|\delta_U\|_{L^2(\O)} \leq \sqrt{ \measv(\Omega)}\sqrt{\overline{U^2}} \leq \frac{M_1}{\sqrt{ \measv(\Omega)}}, \quad \|\delta_W\|_{L^2(\GG)}, \|\delta_W\|_{L^2(\GG)} \leq \frac{M_2}{\sqrt{\measb(\GG)}}. 
		\end{equation*}
		Now we can estimate
		\begin{align*}
		&(\overline{Z} - \overline{U}\,\overline{W})^2\\
		&= \left[(\sqrt{\overline{Z^2}} - R_Z\|\delta_Z\|_{L^2(\GG)}) - (\sqrt{\overline{U^2}} - R_U\|\delta_U\|_{L^2(\O)})(\sqrt{\overline{W^2}} - R_W\|\delta_W\|_{L^2(\GG)}) \right]^2\\
		&\geq \frac{1}{2}\left[\sqrt{\overline{Z^2}}-\sqrt{\overline{U^2}}\sqrt{\overline{W^2}}\right]^2 - C(M_1, M_2)(\|\delta_U\|_{L^2(\O)}^2 + \|\delta_W\|_{L^2(\GG)}^2 + \|\delta_Z\|_{L^2(\GG)}^2).
		\end{align*}
		Therefore, by choosing $\theta\in(0,1)$ small enough we have
		\begin{align*}
		\wt{D}[u,w,z] &\geq C(\|\delta_U\|_{L^2(\O)}^2 + \|\delta_W\|_{L^2(\GG)}^2 + \|\delta_Z\|_{L^2(\GG)}^2 + \theta|\overline{Z} - \overline{U}\,\overline{W}|^2)\\
		&\geq C\theta\left[\sqrt{\overline{Z^2}} -\sqrt{\overline{U^2}}\sqrt{\overline{W^2}}\right]^2
		\end{align*}
		which finishes the proof of {Step 3}.
		
		\medskip
		\noindent\underline{Step 4.} There exists a constant $C>0$ such that
		\begin{equation}\label{eb5}
		\begin{aligned}
		\left(\sqrt{\overline{Z^2}}-\sqrt{\overline{U^2}}\sqrt{\overline{W^2}} \right)^2\geq C\left(|\sqrt{\overline{U^2}} - U_\infty|^2 + |\sqrt{\overline{W^2}} - W_\infty|^2 + |\sqrt{\overline{Z^2}} - Z_\infty|^2\right)
		\end{aligned}
		\end{equation}
		We use the following ansatz
		\begin{equation*}
		\overline{U^2} = U_\infty^2(1+\mu_U)^2, \quad \overline{W^2} = W_\infty^2(1+\mu_W)^2,\quad \overline{Z^2} = Z_\infty^2(1+\mu_Z)^2
		\end{equation*}
		where $\mu_U, \mu_W, \mu_Z \in [-1,\infty)$. It follows that the right hand side of \eqref{eb5} is estimated by
		\begin{equation*}
		C(U_\infty^2\mu_U^2 + W_\infty^2\mu_W^2 + Z_\infty^2\mu_Z^2) \leq C(\mu_U^2+\mu_W^2 + \mu_Z^2)
		\end{equation*}
		while the left hand side of \eqref{eb5} is estimated by
		\begin{equation*}
		\left(Z_\infty(1+\mu_Z) - U_\infty W_\infty(1+\mu_U)(1+\mu_W)\right)^2 = Z_\infty^2\left[(1+\mu_Z) - (1+\mu_U)(1+\mu_W)\right]^2
		\end{equation*}
		where we have used $Z_\infty = \sqrt{\zinf} = \sqrt{\uinf\winf} = U_\infty W_\infty$. Therefore, it remains to show 
		\begin{equation}\label{eb8}
		\left[(1+\mu_Z)-(1+\mu_U)(1+\mu_W)\right]^2 \geq C(\mu_U^2 + \mu_W^2 + \mu_Z^2).
		\end{equation}
		From the conservation laws \eqref{cons1} and \eqref{cons2} we have
		\begin{equation*}
		\measv(\Omega)\overline{U^2} + \measb(\GG)\overline{Z^2} = M_1 =  \measv(\Omega)U_\infty^2 + \measb(\GG)Z_\infty^2
		\end{equation*}
		which implies
		\begin{equation*}
		\measv(\Omega)U_\infty^2(\mu_U^2 + 2\mu_U) + \measb(\GG)Z_\infty^2(\mu_Z^2 + 2\mu_Z) = 0.
		\end{equation*}
		Recalling that $\mu_U, \mu_Z\in [-1,\infty)$, it yields  $\mu_U$ and $\mu_Z$ always have different sign. Similarly, $\mu_W$ and $\mu_Z$ always have different sign, which combined with the previous argument implies that $\mu_U$ and $\mu_W$ always have the same sign. Therefore,
		\begin{align*}
		[(1+\mu_Z) - (1+\mu_U)(1+\mu_W)]^2 &= [\mu_Z - \mu_U - \mu_U(1+\mu_W)]^2\\
		&= [\mu_Z- \mu_U]^2 + 2(\mu_U^2-\mu_Z\mu_U)(1+\mu_W) + [1+\mu_W]^2\\
		&\geq [\mu_Z - \mu_U]^2 \geq \mu_Z^2 + \mu_U^2
		\end{align*}
		since $-\mu_Z\mu_U\geq 0$ and $1+\mu_W \geq 0$. Similarly we have
		\begin{equation*}
		[(1+\mu_Z) - (1+\mu_U)(1+\mu_W)]^2 \geq \mu_Z^2 + \mu_W^2.
		\end{equation*}
		Therefore,
		\begin{equation*}
		[(1+\mu_Z)-(1+\mu_U)(1+\mu_W)]^2 \geq \frac 12[\mu_U^2 + \mu_W^2 + \mu_Z^2],
		\end{equation*}
		which means that \eqref{eb8} is proved and the proof of {\bf Step 4} is finished.
	\end{proof}
	\begin{lemma}[Csisz\'ar-Kullback-Pinsker inequality]\cite{gilardoni2010pinsker}\label{classical-CKP}
		Let $\Omega\subset \R^d$ be a domain and let $f, g \in L^1(\Omega)$ such that $f\ge 0$, $g>0$, and $\intOO f(x) = \intOO g(x) = 1$. Let $\Phi \in C^1([0,\infty))$ satisfy
		\begin{equation}\label{e0}
		\Phi(s)  \ge \Phi(1) + \Phi'(1)(s-1) + \gamma^2(s-1)\mathbf{1}_{\{s\le 1 \}}
		\end{equation}
		for some $\gamma>0$. Then 
		\begin{equation*}
		\|f - g\|_{L^1(\Omega)}^2 \le \frac{4}{\gamma^2}\intOO\left(\Phi\left(\frac fg\right) - \Phi(1) \right)g.
		\end{equation*}
	\end{lemma}
	\begin{lemma}[Csisz\'ar-Kullback-Pinsker-type inequality]\label{CKP}
		For all functions $(u,z,w)$ satisfying the conservation laws \eqref{cons1} and \eqref{cons2}, there exists a constant $C_{CKP}$ depending only on $\measv(\O)$, $\measb(\GG)$ and $(\uinf,\zinf,\winf)$ such that
		\begin{equation*}
		E[u,w,z] \geq C_{CKP}(\|u - \uinf\|_{L^1(\O)}^2 + \|w-\winf\|_{L^1(\GG)}^2  + \|z - \zinf\|_{L^1(\GG)}^2).
		\end{equation*}
	\end{lemma}
	
	\begin{proof}
		We define the averages, thanks to the volume preserving assumption,
		\begin{equation*}
		\ou = \frac{1}{\measv(\O)}\intOO u , \quad \ow = \frac{1}{\measb(\GG)}\intGG w , \quad \oz = \frac{1}{\measb(\GG)}\intGG z .
		\end{equation*}
		Due to the conservation laws \eqref{cons1} and \eqref{cons2} we again have $\ou \le M_1$ and $\oz + \ow\le M_2$ for some $M_1, M_2$ depending only on $\measv(\O)$, $\measb(\GG)$ and $(\uinf,\zinf,\winf)$.
		By direct computations, we have
		\begin{equation}\label{eb1}
		\begin{aligned}
		E[u,w,z] &= \intOO u\log\frac{u}{\ou}  + \intOO u\log\frac{\ou}{\uinf} - u + \uinf \\
		&\quad + \intG w\log\frac{w}{\ow}  + \intGG w\log\frac{\ow}{\winf} - w + \winf  \\
		&\quad + \intGG z\log\frac{z}{\oz}  + \intGG z\log\frac{\oz}{\zinf} - z + \zinf  .
		\end{aligned}
		\end{equation}	
		Let $\Phi(s) = s\log s - s + 1$, it is easy to see that $\Phi$ satisfies \eqref{e0} for all $s\in [0,1]$ with $\gamma = 0.9$. We then can apply Lemma \ref{classical-CKP} to $f(x) = \measv(\O)^{-1}u(x)/\ou$ and $g(x) \equiv \measv{(\O)}^{-1}$ to obtain
		\begin{equation*}
		\measv(\Omega)^{-2}\left\|\frac{u}{\ou}-1\right\|_{L^1(\Omega)}^2 \le 
		\frac{4}{0.9^2}\measv(\Omega)^{-1}\intOO \left(\frac{u}{\ou}\log \frac{u}{\ou} - \frac{u}{\ou} + 1 \right)\ou
		\end{equation*}
		which leads to
		\begin{equation*}
		\|u - \ou\|_{L^1(\Omega)}^2 \le \frac{4\ou^2\measv(\O)}{0.9^2}\intOO u\log\frac{u}{\ou} \le \frac{4M_1^2\measv(\O)}{0.9^2}\intOO u\log\frac{u}{\ou}.
		\end{equation*}
		Similarly, we have
		\begin{equation*}
		\|z - \oz\|_{L^1(\Gamma)}^2 \le \frac{4M_2^2\measb(\GG)}{0.9^2}\intGG z\log\frac{z}{\oz} \quad \text{and} \quad \|w - \ow\|_{L^1(\Gamma)}^2 \le \frac{4M_2^2\measb(\GG)}{0.9^2}\intGG w\log\frac{w}{\ow}. 
		\end{equation*}
		By using the elementary inequality $x\log(x/y) - x + y\geq (\sqrt x - \sqrt y)^2$ we have
		\begin{align*}
		\intOO u\log\frac{\ou}{\uinf} - \ou + \uinf   &= \measv(\O)\left(\ou\log\frac{\ou}{\uinf} - \ou + \uinf\right)\\
		&\geq \measv(\O)(\sqrt{\ou} - \sqrt{\uinf})^2\\
		&= \measv(\O)\frac{|\ou - \uinf|^2}{(\sqrt{\ou} + \sqrt{\uinf})^2}\\
		&\geq \frac{1}{\measv(\O)(\sqrt{M_1} + \sqrt{\uinf})^2}\|\ou - \uinf\|_{L^1(\O)}^2
		\end{align*}
		and similarly
		\begin{equation*}
		\intGG w\log\frac{\ow}{\winf} - w + \winf   \geq \frac{1}{\measb(\GG)(\sqrt{M_2}+\sqrt{\winf})^2}\|\ow - \winf\|_{L^1(\GG)}^2
		\end{equation*}
		and
		\begin{equation*}
		\intGG z\log\frac{\oz}{\zinf} - z + \zinf   \geq \frac{1}{\measb(\GG)(\sqrt{M_2}+\sqrt{\zinf})^2}\|\oz - \zinf\|_{L^1(\GG)}^2
		\end{equation*}
		Inserting these estimates into \eqref{eb1} we get the desired result.
	\end{proof}				
	
	\subsection{Proof of Theorem \ref{thm:convergence}}
	By direct computations, the system \eqref{sys} possesses the following linearly independent conservation laws for all $t>0$
	\begin{equation}\label{conserv1}
	\intO u(x,t)  + \intG z(x,t)  = M_1:= \int_{\O_0} u_0(x)  + \int_{\GG_0} z_0(x) 
	\end{equation}
	and
	\begin{equation}\label{conserv2}
	\intG w(x,t)  + \intG z(x,t)  = M_2:= \int_{\GG_0} w_0(x)  + \int_{\GG_0} z_0(x) .
	\end{equation}
	\begin{lemma}\label{lem:equilibrium}
		Fix an initial mass $(M_1,M_2)\in (0,\infty)^2$. There exists a unique positive solution $(\uinf,\zinf,\winf)\in (0,\infty)^3$ to the algebraic system
		\begin{equation*}
		\begin{cases}
		\zinf = \uinf\winf,\\
		\measv(\Omega_0)\uinf + \measb(\Gamma_0)\zinf = M_1,\\
		\measb(\Gamma_0)\winf + \measb(\Gamma_0)\zinf = M_2.
		\end{cases}
		\end{equation*}
	\end{lemma}
	\begin{proof}
		The proof of this lemma follows from direct computations so we omit it.
	\end{proof}
	We are now ready to prove the results of convergence to equilibrium for \eqref{sys}.
	\begin{proof}[Proof of Theorem \ref{thm:convergence}]
		We consider the relative entropy
		\begin{equation}\label{entropy}
		E[u,w,z] = \intO u\log\frac{u}{\uinf} - u + \uinf   + \intG w\log\frac{w}{\winf}-w+\winf   + \intG z\log\frac{z}{\zinf}-z+\zinf .
		\end{equation}
		By direct computation we obtain the entropy production
		\begin{equation}\label{en_prod}
		\begin{aligned}
		&D[u,w,z] = -\frac{d}{dt}E[u,w,z]\\
		&= \dO\intO\frac{|\na u|^2}{u}  + \dG\intG\frac{|\naG w|^2}{w}  + \dGG\intG\frac{|\naG z|^2}{z}  + \intG(z - uw)\log\frac{z}{uw} \\
		&\quad - {\color{black}\intG \left[u\log\frac{u}{\uinf}\right](\VV_p-\VV_\GG)\cdot \nu  } - \intO(\VV_\O - \VV_p)\cdot \na u   + \intO(u-\uinf)\na\cdot \VV_p \\
		&\quad -\intG \VV_\GG \cdot \naG w   - \intG \VV_\GG\cdot \naG z . 
		\end{aligned}
		\end{equation}
		By the assumption $(\VV_p - \VV_\Gamma)\cdot \nu = 0$ on $\Gamma(t)$, the fifth term on the right hand side vanishes. For the sixth and seventh terms on the right hand side we can calculate further as
		\begin{align*}
		&-\intO (\VV_\Omega - \VV_p)\cdot \na u   + \intO(u-\uinf)\na\cdot \VV_p \\
		& = -\intO \VV_\Omega\cdot \na u   + \intO \na\cdot((u-\uinf)\VV_p) \\
		&= -\intO \VV_\Omega\cdot \na u   + \intG (u-\uinf)\VV_p\cdot \nu  \\
		&= -\intO \VV_\Omega\cdot \na u   + \intG (u-\uinf)\VV_\Gamma\cdot \nu  
		\end{align*}
		where we used again the assumption $\VV_p\cdot \nu = \VV_\Gamma\cdot \nu = 0$ on $\Gamma(t)$. By Cauchy-Schwarz's inequality and the conservation of mass we can estimate
		\begin{align*}
		-\intO\VV_\O \cdot \na u   &= -\intO \VV_\O \sqrt{u}\frac{\na u}{\sqrt u} \\&\geq -C\|\VV_\O\|_{\infty}^2\intO u(t)  - \frac{\dO}{4}\intO\frac{|\na u|^2}{u} \\
		&\geq -C\|\VV_\O\|_{\infty}^2 - \frac{\dO}{4}\intO \frac{|\na u|^2}{u} ,
		\end{align*}
		and similarly
		\begin{equation*}
		-\intG\VV_\GG\cdot\naG w  \geq C\|\VV_\GG\|_\infty^2 - \frac{\dG}{2}\intG\frac{|\naG w|^2}{w} 
		\end{equation*}
		and
		\begin{equation*}
		-\intG \VV_\GG\cdot\naG z   \geq C\|\VV_\GG\|_{\infty}^2 - \frac{\dGG}{2}\intG\frac{|\naG z|^2}{z} .
		\end{equation*}
		It remains to deal with the term $\intG (u-\uinf)\VV_\Gamma\cdot \nu  $. It is obvious
		\begin{equation*}
		-\intG \uinf \VV_\Gamma\cdot \nu   \geq -\uinf |\Gamma(t)|\|\VV_\Gamma\|_{\infty} = -C\|\VV_\Gamma\|_{\infty}
		\end{equation*}
		thanks to the assumption \eqref{B}. Finally, thanks to the trace interpolation inequality 
		\begin{align*}
		\intG u \VV_\Gamma \cdot \nu   &\geq -\|\VV_\Gamma\|_{\infty}\intG |\sqrt u|^2 \\
		&\geq -C\|\VV_\Gamma\|_{\infty}\|\sqrt{u}\|_{L^2(\Omega(t))}\left(\|\na \sqrt u\|_{L^2(\Omega(t))} + \|\sqrt u\|_{L^2(\Omega(t))} \right)\\
		&\geq -\delta_\Omega\|\na \sqrt u\|_{L^2(\Omega(t))}^2  - C\|\VV_\Gamma\|_{\infty}^2\|\sqrt u\|_{L^2(\Omega(t))}^2\\
		&\geq -\frac{\delta_\Omega}{4}\intO\frac{|\na u|^2}{u}  - C\|\VV_\Gamma\|_{\infty}^2
		\end{align*}
		Inserting all these estimates into \eqref{en_prod} we have
		\begin{equation*}
		D[u,w,z] \geq \wt{D}[u,w,z] - C(\|\VV_\O\|_\infty^2 + \|\VV_\GG\|_\infty^2 + \|\VV_\Gamma\|_\infty)
		\end{equation*}
		where
		\begin{equation}\label{eq:dtilde}
		\wt{D}[u,w,z] = \frac \dO 2\intO\frac{|\na u|^2}{u}  + \frac\dG 2\intG\frac{|\naG w|^2}{w}  + \frac\dGG 2\intG\frac{|\naG z|^2}{z}  + \intG(z-uw)\log\frac{z}{uw} .
		\end{equation}
		From the entropy-entropy dissipation estimate in Proposition \ref{pro:eed}, we have
		\begin{equation*}
		\widetilde{D}[u,w,z]\ge \lambda E[u,w,z]
		\end{equation*}
		where, thanks to the assumptions \eqref{B}, \eqref{aLSI} and \eqref{aTrPW}, the constant $\lambda>0$ can be chosen \textit{independent of time $t>0$}.
		Consequently,
		\begin{equation*}
		\frac{d}{dt}E[u,w,z] +\lambda E[u,w,z] \leq C(\|\VV_\O\|_\infty^2 + \|\VV_\GG\|_\infty^2 + \|\VV_\Gamma\|_\infty).
		\end{equation*}
		Let $h(t):= \|\VV_\O(t)\|_\infty^2 + \|\VV_\GG(t)\|_\infty^2 + \|\VV_\Gamma(t)\|_\infty$, we have $\limsup_{t\to\infty}h(t) = 0$. Using Gronwall's lemma we have
		\begin{equation}\label{ea1}
		E[u,w,z](t) \leq e^{-\lambda t}E[u,w,z](0) + e^{-\lambda t}\int_0^te^{\lambda s}h(s) .
		\end{equation}
		The first part on the right hand side decays exponentially to zero. For the second part, let $\varepsilon>0$ arbitrary. There exists $T>0$ such that $h(s) \leq \varepsilon\lambda/2$ for all $s\geq T$. We can then estimate for $t>T$,
		\begin{align*}
		e^{-\lambda t}\int_0^te^{\lambda s}h(s)  &\leq e^{-\lambda (t-T)}\int_0^Th(s)  + \varepsilon e^{-\lambda t}\int_T^te^{\lambda s} \\
		&\leq e^{-\lambda (t-T)}\int_0^Th(s)  + \frac{\varepsilon\lambda}{2} e^{-\lambda t}\frac{e^{\lambda t} - e^{\lambda T}}{\lambda}\\
		&\leq e^{-\lambda (t-T)}\int_0^Th(s)  +\frac{\varepsilon}{2}.
		\end{align*}
		Since $T$ now is fixed, there exists $T_1>T$ such that
		\begin{equation*}
		e^{-\lambda(t-T)}\int_0^Th(s)  \leq \frac\varepsilon 2 \quad \text{ for all } \quad t\geq T_1.
		\end{equation*}
		Therefore, for all $t\geq T_1$,
		\begin{equation*}
		e^{-\lambda t}\int_0^te^{\lambda s}h(s)  \leq \varepsilon
		\end{equation*}
		and thus
		\begin{equation*}
		\limsup_{t\to\infty}e^{-\lambda t}\int_0^te^{\lambda s}h(s)  \leq\varepsilon
		\end{equation*}
		and since $\varepsilon>0$ arbitrary, we finally obtain
		\begin{equation*}
		\limsup_{t\to\infty}e^{-\lambda t}\int_0^te^{\lambda s}h(s)  = 0.
		\end{equation*}
		Inserting this into \eqref{ea1} we get
		\begin{equation*}
		\lim_{t\to\infty}E[u,w,z](t) = 0.
		\end{equation*}
		By the Csisz\'ar-Kullback-Pinsker in Lemma \ref{CKP} we obtain
		\begin{equation*}
		\|u(t) - \uinf\|_{L^1(\Omega(t))} + \|w(t) - \winf\|_{L^1(\GG(t))} + \|z(t) - \zinf\|_{L^1(\GG(t))} \to 0 \quad \text{ as } \quad t\to\infty.
		\end{equation*}
		If \eqref{ex_decay} holds, then it follows directly from \eqref{ea1} and Lemma \ref{CKP} that
		\begin{equation*}
		\|u(t) - \uinf\|_{L^1(\Omega(t))} + \|w(t) - \winf\|_{L^1(\GG(t))} + \|z(t) - \zinf\|_{L^1(\GG(t))} \leq Ce^{-\mu t}
		\end{equation*}
		for some $\mu >0$. 
	\end{proof}
	
	\bibliographystyle{alpha}

\end{document}